\documentclass[11pt]{amsart}
\usepackage{amssymb,bbold,mathtools}
\usepackage[all]{xy}
\usepackage{color}
\usepackage{mathrsfs}
\usepackage{mathtools}
\usepackage{enumitem}
\usepackage{hyperref}
\usepackage{cleveref}

\usepackage[charter]{mathdesign}

\allowdisplaybreaks

\usepackage{tikz-cd}
\usetikzlibrary{decorations.markings}
\tikzset{negated/.style={
		decoration={markings,
			mark= at position 0.5 with {
				\node[transform shape] (tempnode) {$\backslash$};
			}
		},
		postaction={decorate}
	}
}

\DeclareMathAlphabet{\mathcalalt}{OMS}{cmsy}{m}{n}

\theoremstyle{plain}
\newtheorem{theorem}{Theorem}[section]
\newtheorem{proposition}[theorem]{Proposition}
\newtheorem{corollary}[theorem]{Corollary}
\newtheorem{lemma}[theorem]{Lemma}

\theoremstyle{definition}
\newtheorem{remark}[theorem]{Remark}
\newtheorem{example}[theorem]{Example}
\newtheorem{definition}[theorem]{Definition}

\newcommand{\abs}[1]{\lvert#1\rvert}
\newcommand{\norm}[1]{\lVert#1\rVert}

\newcommand{\uppars}[1]{\textup{(}#1\textup{)}}

\newcommand{\EE}{\mathbb E}
\newcommand{\RR}{\mathbb R}
\newcommand{\NN}{\mathbb N}
\newcommand{\ZZ}{\mathbb Z}

\newcommand{\pos}[1]{{#1}^+}
\newcommand{\posE}{\pos{E}}

\newcommand{\ob}{\mathrm{ob}}
\newcommand{\oc}{\mathrm{oc}}
\newcommand{\reg}{\mathrm r}
\newcommand{\obops}{\mathcal{L}_{\ob}}
\newcommand{\ocops}{\mathcal{L}_{\oc}}
\newcommand{\regops}{\mathcal{L}_{\ob}}

\newcommand{\optop}[1]{\mathrm{ASOT}_{#1}}

\newcommand{\opfont}{\mathcal}

\newcommand{\opgroup}{{\opfont G}}
\newcommand{\oplin}{{\opfont V}}
\newcommand{\oplat}{{\opfont E}}
\newcommand{\opset}{{\opfont S}}
\newcommand{\alg}{\opfont A}
\newcommand{\algtwo}{\opfont B}
\newcommand{\opideal}{\opfont I}
\newcommand{\centre}{{\opfont Z}}

\newcommand{\odual}[1]{{#1}^{\thicksim}}
\newcommand{\ocdual}[1]{{#1}^\thicksim_{\oc}}

\newcommand{\Ell}{\mathrm{L}}

\newcommand{\dc}{\mathrm{d}}

\newcommand{\conv}{\xrightarrow}
\newcommand{\convwithoverset}[1]{\conv{#1}}

\renewcommand{\o}{\mathrm{o}}

\newcommand{\uo}{\mathrm{uo}}

\newcommand{\unb}{\mathrm{u}}

\newcommand{\uoLt}{{\widehat{\tau}}}

\newcommand{\oconv}{\convwithoverset{\o}}

\newcommand{\uoconv}{\convwithoverset{\uo}}

\newcommand{\tauconv}{\convwithoverset{\tau}}

\newcommand{\netgen}[2]{{(#1)_{#2}}}

\newcommand{\indexsetfont}{\mathcalalt}
\newcommand{\is}{\indexsetfont{A}}
\newcommand{\net}{\netgen{x_\alpha}{\alpha\in\is}}
\newcommand{\opnet}{\netgen{T_\alpha}{\alpha\in\is}}

\newcommand{\istwo}{\indexsetfont{B}}
\newcommand{\nettwo}{\netgen{x_\beta}{\beta\in\istwo}}

\newcommand{\seq}[2]{(#1)_{#2=1}^\infty}

\newcommand{\oLtop}{o-Lebes\-gue topology}
\newcommand{\oLtops}{o-Lebes\-gue topologies}
\newcommand{\oL}{o-Lebesgue}
\newcommand{\uoLtop}{uo-Lebes\-gue topology}
\newcommand{\uoLtops}{uo-Lebes\-gue topologies}

\newcommand{\necun}{\textup{(}necessarily unique\textup{)}}

\DeclareMathOperator{\Orth}{Orth}
\newcommand{\Id}{I}

\usepackage{datetime}
\usepackage{currfile}
\newcommand*\patchAmsMathEnvironmentForLineno[1]{%
	\expandafter\let\csname old#1\expandafter\endcsname\csname #1\endcsname
	\expandafter\let\csname oldend#1\expandafter\endcsname\csname end#1\endcsname
	\renewenvironment{#1}%
	{\linenomath\csname old#1\endcsname}%
	{\csname oldend#1\endcsname\endlinenomath}}%
\newcommand*\patchBothAmsMathEnvironmentsForLineno[1]{%
	\patchAmsMathEnvironmentForLineno{#1}%
	\patchAmsMathEnvironmentForLineno{#1*}}%
\AtBeginDocument{%
	\patchBothAmsMathEnvironmentsForLineno{equation}%
	\patchBothAmsMathEnvironmentsForLineno{align}%
	\patchBothAmsMathEnvironmentsForLineno{flalign}%
	\patchBothAmsMathEnvironmentsForLineno{alignat}%
	\patchBothAmsMathEnvironmentsForLineno{gather}%
	\patchBothAmsMathEnvironmentsForLineno{multline}%
}

\crefname{theorem}{Theorem}{Theorems}
\crefname{proposition}{Proposition}{Propositions}
\crefname{lemma}{Lemma}{Lemmas}
\crefname{corollary}{Corollary}{Corollaries}
\crefname{conjecture}{Conjecture}{Conjectures}
\crefname{definition}{Definition}{Definitions}
\crefname{example}{Example}{Examples}
\crefname{remark}{Remark}{Remarks}
\crefname{assumption}{Assumption}{Assumptions}
\crefname{hypothesis}{Hypothesis}{Hypotheses}
\crefname{question}{Question}{Questions}
\crefname{problem}{Problem}{Problems}
\crefname{task}{Task}{Tasks}
\crefname{addendum}{Addendum}{Addenda}
\crefname{idea}{Idea}{Ideas}
\crefname{suggestion}{Suggestion}{Suggestions}
\crefname{context}{Context}{Contexts}
\crefname{exercise}{Exercise}{Exercises}
\crefname{section}{Section}{Sections}
\crefname{subsection}{Section}{Sections}
\crefname{subsubsection}{Section}{Sections}

\crefname{equation}{equation}{equations}
\crefname{enumi}{part}{parts}
\crefname{enumii}{part}{parts}
\crefname{enumiii}{part}{parts}
\crefname{enumiv}{part}{parts}

\newcommand{\enclosepart}[1]{(#1)}
\newcommand{\partref}[1]{\enclosepart{\ref{#1}}}

\usepackage[mathlines]{lineno}
\makeatletter
\def\slashedarrowfill@#1#2#3#4#5{%
	$\m@th\thickmuskip0mu\medmuskip\thickmuskip\thinmuskip\thickmuskip
	\relax#5#1\mkern-7mu%
	\cleaders\hbox{$#5\mkern-2mu#2\mkern-2mu$}\hfill
	\mathclap{#3}\mathclap{#2}%
	\cleaders\hbox{$#5\mkern-2mu#2\mkern-2mu$}\hfill
	\mkern-7mu#4$%
}

\def\rightslashedarrowfilla@{%
	\slashedarrowfill@\relbar\relbar{\raisebox{1.2pt}{$\scriptscriptstyle\diagup$}}\rightarrow}
\newcommand\xslashedrightarrowa[2][]{%
	\ext@arrow 0055{\rightslashedarrowfilla@}{#1}{#2}}

\def\rightslashedarrowfillb@{%
	\slashedarrowfill@\relbar\relbar/\rightarrow}
\newcommand\xslashedrightarrowb[2][]{%
	\ext@arrow 0055{\rightslashedarrowfillb@}{#1}{#2}}

\def\rightslashedarrowfillc@{%
	\slashedarrowfill@\relbar\relbar{\raisebox{.12em}{\tiny/}}\rightarrow}
\newcommand\xslashedrightarrowc[2][]{%
	\ext@arrow 0055{\rightslashedarrowfillc@}{#1}{#2}}

\pgfdeclareshape{slash underlined}
{
	\inheritsavedanchors[from=rectangle] % this is nearly a circle
	\inheritanchorborder[from=rectangle]
	\inheritanchor[from=rectangle]{north}
	\inheritanchor[from=rectangle]{north west}
	\inheritanchor[from=rectangle]{north east}
	\inheritanchor[from=rectangle]{center}
	\inheritanchor[from=rectangle]{west}
	\inheritanchor[from=rectangle]{east}
	\inheritanchor[from=rectangle]{mid}
	\inheritanchor[from=rectangle]{mid west}
	\inheritanchor[from=rectangle]{mid east}
	\inheritanchor[from=rectangle]{base}
	\inheritanchor[from=rectangle]{base west}
	\inheritanchor[from=rectangle]{base east}
	\inheritanchor[from=rectangle]{south}
	\inheritanchor[from=rectangle]{south west}
	\inheritanchor[from=rectangle]{south east}
	\inheritanchorborder[from=rectangle]
	\foregroundpath{
		\southwest \pgf@xa=\pgf@x \pgf@ya=\pgf@y
		\northeast \pgf@xb=\pgf@x \pgf@yb=\pgf@y
		\pgf@xc=\pgf@xa
		\advance\pgf@xc by .5\pgf@xb
		\pgf@yc=\pgf@ya
		\advance\pgf@xc by -1.3pt
		\advance\pgf@yc by -1.8pt
		\pgfpathmoveto{\pgfqpoint{\pgf@xc}{\pgf@yc}}
		\advance\pgf@xc by  2.6pt
		\advance\pgf@yc by  3.6pt
		\pgfpathlineto{\pgfqpoint{\pgf@xc}{\pgf@yc}}
		\pgfpathmoveto{\pgfqpoint{\pgf@xa}{\pgf@ya}}
		\pgfpathlineto{\pgfqpoint{\pgf@xb}{\pgf@ya}}
	}
}
\tikzset{nomorepostaction/.code=\let\tikz@postactions\pgfutil@empty}

\makeatother

\begin{document}
\title[Convergence structures and locally solid topologies]{Convergence structures and locally solid topologies on vector lattices of operators}

\author{Yang Deng}
\address{Yang Deng, School of Economic Mathematics, Southwestern University of Finance and Economics, Chengdu, Sichuan, 611130, PR China} 
\email{dengyang@swufe.edu.cn}

\author{Marcel de Jeu}
\address{Marcel de Jeu, Mathematical Institute, Leiden University, P.O.\ Box 9512, 2300 RA Leiden, the Netherlands;
	and Department of Mathematics and Applied Mathematics, University of Pretoria, Cor\-ner of Lynnwood Road and Roper Street, Hatfield 0083, Pretoria, South Africa}
\email{mdejeu@math.leidenuniv.nl}

%\thanks{During this research, the first author was supported by a grant of China Scholarship Council (CSC)}

%\makeatletter{\renewcommand*{\@makefnmark}{}
%	\date{}\footnote{{File: \currfilebase. Compiled: \today,\,\currenttime.}}
%	\makeatother}

\begin{abstract}For vector lattices $E$ and $F$, where $F$ is Dedekind complete and supplied with a locally solid topology, we introduce the corresponding locally solid absolute strong operator topology on the order bounded operators $\mathcal L_{\mathrm{ob}}(E,F)$ from $E$ into $F$. Using this, it follows that $\mathcal L_{\mathrm{ob}}(E,F)$ admits a Hausdorff uo-Lebesgue topology whenever $F$ does.\\
For each of order convergence, unbounded order convergence, and\textemdash when applicable\textemdash convergence in the Hausdorff uo-Lebesgue topology, there are both a uniform and a strong convergence structure on $\mathcal L_{\mathrm {ob}}(E,F)$. Of the six conceivable inclusions within these three pairs, only one is generally valid. On the orthomorphisms of a Dedekind complete vector lattice, however, five are generally valid, and the sixth is valid for order bounded nets. The latter condition is redundant in the case of \emph{sequences} of orthomorphisms, as a consequence of a uniform order boundedness principle for orthomorphisms that we establish.
\\
We furthermore show that, in contrast to general order bounded operators, orthomorphisms preserve not only order convergence of nets, but unbounded order convergence and\textemdash when applicable\textemdash conver\-gence in the Hausdorff uo-Lebesgue topology as well.
\end{abstract}

\maketitle

\section{Introduction and overview}\label{sec:introduction_and_overview}

\noindent Let $X$ be a non-empty set. A \emph{convergence structure on $X$} is a non-empty collection $\mathcal C$ of pairs $(\net,x)$, where $\net$ is a net in $X$ and $x\in X$, such that:
\begin{enumerate}
\item when $(\net,x)\in\mathcal C$, then also $(\nettwo,x)\in\mathcal C$ for every subnet $\nettwo$ of $\net$;
\item when a net $\net$ in $X$ is constant with value $x$, then $(\net,x)\in\mathcal C$.
\end{enumerate}
One can easily vary on this definition. For example, one can allow only sequences. There does not appear to be a consensus in the literature about the notion of a convergence structure; \cite{beattie_butzmann_CONVERGENCE_STRUCTURES_AND_APPLICATIONS_TO_FUNCTIONAL_ANALYSIS:2002} uses filters, for example. Ours is sufficient for our merely descriptive purposes, and close in spirit to what may be the first occurrence of such a definition in \cite{dudley:1964} for sequences. Although we shall not pursue this in the present paper, let us still mention that the inclusion of the subnet criterion in the definition makes it possible to introduce an associated topology on $X$ in a natural way. Indeed, define a subset of $S$ of $X$ to be \emph{$\mathcal C$-closed} when $x\in S$ for all pairs $(\net,x)\in\mathcal C$ such that $\net\subseteq S$. Then the collection of the complements of the {$\mathcal C$-closed} subsets of $X$ is a topology on $X$.

The convergent nets in a topological space, together with their limits, are the archetypical example of a convergence structure. For a given convergence structure $\mathcal C$ on a non-empty set $X$, however, it is not always possible to find a (obviously unique) topology $\tau$ on $X$ such that the $\tau$-convergent nets in $X$, together with their limits, are precisely the elements of $\mathcal C$. Such non-topological convergence structures arise naturally in the context of vector lattices. For example, the order convergent nets in a vector lattice, together with their order limits, form a convergence structure, but this convergence structure is topological if and only if the vector lattice is finite dimensional; see \cite[Theorem~1]{dabboorasad_emelyanov_marabeh_UNPUBLISHED:2017} or \cite[Theorem~8.36]{taylor_THESIS:2018}. Likewise, the unbounded order convergent nets in a vector lattice, together with their unbounded order limits, form a convergence structure, but this convergence structure is topological if and only if the vector lattice is atomic; see \cite[Theorem~6.54]{taylor_THESIS:2018}. Topological or not, the order and unbounded order convergence structures, together with the (topological) structure for convergence in the Hausdorff \uoLtop, when this exists, yield three natural and related convergence structures on a vector lattice to consider.

Suppose that $E$ and $F$ are vector lattices, where $F$ is Dedekind complete. The above then yields three convergence structures on the vector lattice $\obops(E,F)$ of order bounded operators from $E$ into $F$. On the other hand, there are also three convergence structures on $\obops(E,F)$ that are naturally derived from the three convergence structures on the vector lattice $F$. For example, one can consider all pairs $(\opnet,T)$, where $\opnet$ is a net in $\obops(E,F)$ and $T\in\obops(E)$, such that $\netgen{T_\alpha x}{\alpha\in\is}$ is order convergent to $Tx$ in $F$ for all $x\in E$. These pairs also form a convergence structure on $\obops(E,F)$. Likewise, the pointwise unbounded order convergence in $F$ and\textemdash when applicable\textemdash the pointwise convergence in the Hausdorff \uoLtop\ on $F$ both yield a convergence structure on $\obops(E,F)$. Motivated by the terminology for operators between Banach spaces, we shall speak of \emph{uniform} and \emph{strong} convergence structures on $\obops(E)$\textemdash with the obvious meanings.

The present paper is primarily concerned with the possible inclusions between the uniform and strong convergence structure for each of order convergence, unbounded order convergence, and\textemdash when applicable\textemdash convergence in the Hausdorff \uoLtop. Is it true that a uniformly order convergent net of order bounded operators is also strongly order convergent? Is the converse true? How is this for unbounded order convergence and, when applicable, convergence in the Hausdorff \uoLtop? We consider these implications, six in all, for $\obops(E,F)$, but also for the orthomorphisms $\Orth(E)$ on a Dedekind complete vector lattice.\footnote{With six convergence structures under consideration, one can actually consider thirty non-trivial possible inclusions between them. With some more effort, one can determine for all of these whether they are generally valid for the order bounded operators and for the orthomorphisms on a Dedekind complete vector lattice; see \cite[Tables~3.1 and~3.2]{deng_de_jeu_UNPUBLISHED:2020c}.} This special interest in $\Orth(E)$ stems from representation theory. When a group acts as order automorphisms on a Dedekind complete vector lattice $E$, then the Boolean lattice of all invariant bands in $E$ can be retrieved from the commutant of the group action in $\Orth(E)$. This commutant, therefore, plays the role of the von Neumann algebra which is the commutant of a unitary action of a group on a Hilbert space. It has been known long since that more than one topology on a von Neumann algebra is needed to understand it and its role in representation theory on Hilbert spaces, and the same holds true for the convergence structures as related to these commutants in an ordered context. Using these convergence structures, it is, for example, possible to obtain ordered versions of von Neumann's bicommutant theorem. We shall report separately on this. Apart from its intrinsic interest, the material on $\Orth(E)$ in the present paper is an ingredient for these next steps.

\smallskip
This paper is organised as follows.
\smallskip

\cref{sec:preliminaries} contains the basic notations, definitions, conventions, and references to earlier results.

In \cref{sec:topologies_on_vector_lattices_of_order_bounded_operators}, we show how, given a vector lattice $E$, a Dedekind complete vector lattice $F$, and a (not necessarily Hausdorff) locally solid linear topology $\tau_F$ on $F$, a locally solid linear topology can be introduced on $\obops(E,F)$ that deserves to be called the absolute strong operator topology that is generated by $\tau_F$. This is a preparation for \cref{sec:lebesgue_topologies_on_vector_lattices_of_operators}, where we show that regular vector sublattices of $\obops(E,F)$ admit a Hausdorff \uoLtop\ when $F$ admits one.

For each of order convergence, unbounded order convergence, and\textemdash when applicable\textemdash convergence in the Hausdorff \uoLtop, there are two conceivable implications between uniform and strong convergence of a net of order bounded operators. In \cref{sec:uniform_and_strong_convergence_structures_on_regops}, we show that only one of these six is generally valid. \cref{sec:uniform_and_strong_convergence_structures_on_orth} will make it clear that the five  failures are, perhaps, not as `only to be expected' as one might think at first sight.

In \cref{sec:orthomorphisms}, we review some material concerning orthomorphism and establish a few auxiliary result for use in the present paper and in future ones. It is shown here that a Dedekind complete vector lattice and its orthomorphisms have the same universal completion. Furthermore, a uniform order boundedness principle is established for sets of orthomorphisms. 

\cref{sec:continuity_of_orthomorphisms} briefly digresses from the main line of the paper. It is shown that orthomorphisms preserve not only the order convergence of nets, but also the unbounded order convergence and\textemdash when applicable\textemdash the convergence in the Hausdorff \uoLtop. None of this is true for arbitrary order bounded operators.

In \cref{sec:topologies_on_the_orthomorphisms_of_a_vector_lattice}, we return to the main line, and we specialise the results in \cref{sec:topologies_on_vector_lattices_of_order_bounded_operators,sec:lebesgue_topologies_on_vector_lattices_of_operators} to the orthomorphisms. When restricted to $\Orth(E)$, the absolute strong operator topologies from \cref{sec:topologies_on_vector_lattices_of_order_bounded_operators} are simply strong operator topologies.

\cref{sec:uniform_and_strong_convergence_structures_on_orth} on orthomorphisms is the companion of  \cref{sec:uniform_and_strong_convergence_structures_on_regops}, but the results are quite in contrast. For each of order convergence, unbounded order convergence, and\textemdash when applicable\textemdash convergence in the  Hausdorff \uoLtop, both  implications between uniform and strong convergence of a net of orthomorphisms \emph{are} valid, with an order boundedness condition on the net being necessary only for order convergence. For \emph{sequences} of orthomorphisms, this order boundedness condition is even redundant as a consequence of the uniform order boundedness principle for orthomorphisms from \cref{sec:orthomorphisms}.

\section{Preliminaries}\label{sec:preliminaries}

\noindent In this section, we collect a number of definitions, notations, conventions and earlier results.

All vector spaces are over the real numbers; all vector lattices are supposed to be Archimedean. We write $\posE$ for the positive cone of a vector lattice $E$. For a non-empty subset $S$ of $E$, we let $I_S$ and $B_S$ denote the ideal of $E$ and the band in $E$, respectively, that are generated by $S$; we write $S^\vee$ for $\{\,s_1\vee\dotsb\vee s_n: s_1,\dotsc,s_n\in S,\,n\geq 1\, \}$.

Let $E$ be a vector lattice, and let $x\in E$. We say that a net $\net$ in $E$ is \emph{order convergent to $x\in E$} (denoted by $x_\alpha\oconv x$) when there exists a net $(y_\beta)_{\beta\in \istwo}$ in $E$ such that $y_\beta\downarrow 0$ and with the property that, for every $\beta_0\in\istwo$, there exists an $\alpha_0\in \is$ such that $\abs{x-x_\alpha}\leq y_{\beta_0}$ whenever $\alpha$ in $\is$ is such that $\alpha\geq\alpha_0$. We explicitly include this definition to make clear that the index sets $\is$ and $\istwo$ need not be equal.

Let $\net$ be a net in a vector lattice $E$, and let $x\in E$. We say that $(x_\alpha)$ is \emph{unbounded order convergent to $x$ in $E$}  (denoted by $x_\alpha\uoconv x$) when $\abs{x_\alpha-x}\wedge y\oconv 0$ in $E$ for all $y\in \posE$. Order convergence implies unbounded order convergence to the same limit. For order bounded nets, the two notions coincide.

Let $E$ and $F$ be vector lattices. The order bounded operators from $E$ into $F$ will be denoted by $\obops(E,F)$. We write $\odual{E}$ for $\obops(E,\RR)$. A linear operator  $T: E\to F$ between two vector lattices $E$ and $F$ is \emph{order continuous} when, for every net $\net$ in $E$,  the fact that $x_\alpha\oconv 0$ in $E$ implies that $Tx_\alpha\oconv 0$ in $F$. An order continuous linear operator between two vector lattices is automatically order bounded; see \cite[Lemma~1.54]{aliprantis_burkinshaw_POSITIVE_OPERATORS_SPRINGER_REPRINT:2006}, for example. The order continuous linear operators from $E$ into $F$ will be denoted by $\ocops(E,F)$. We write $\ocdual{E}$ for $\ocops(E,\RR)$.

Let $F$ be a vector sublattice of a vector lattice $E$. Then $F$ is a \emph{regular vector sublattice of $E$} when the inclusion map from $F$ into $E$ is order continuous. Ideals are regular vector sublattices. For a net in a regular vector sublattice $F$ of $E$, its uo-convergence in $F$ and in $E$ are equivalent; see \cite[Theorem~3.2]{gao_troitsky_xanthos:2017}.

When $E$ is a vector space, a \emph{linear topology on $E$} is a (not necessarily Hausdorff) topology that provides $E$ with the structure of a topological vector space. When $E$ is a vector lattice, a \emph{locally solid linear topology on $E$} is a linear topology on $E$ such that there exists a base of (not necessarily open) neighbourhoods of 0 that are solid subsets of $E$.
For the general theory of locally solid linear topologies on vector lattices we refer to \cite{aliprantis_burkinshaw_LOCALLY_SOLID_RIESZ_SPACES_WITH_APPLICATIONS_TO_ECONOMICS_SECOND_EDITION:2003}.
When $E$ is a vector lattice, a \emph{locally solid additive topology on $E$} is a topology that provides the additive group $E$ with the structure of a (not necessarily Hausdorff) topological group, such that there exists a base of (not necessarily open) neighbourhoods of 0 that are solid subsets of $E$.

A topology $\tau$ on a vector lattice $E$ is an \emph{\oLtop} when it is a (not necessarily Hausdorff) locally solid linear topology on $E$ such that, for a net $\net$ in $E$, the fact that $x_\alpha\oconv 0$ in $E$ implies that $x_\alpha\tauconv 0$. A vector lattice need not admit a Hausdorff \oLtop. A topology $\tau$ on a vector lattice $E$ is a \emph{\uoLtop} when it is a (not necessarily Hausdorff) locally solid linear topology on $E$ such that, for a net $\net$ in $E$, the fact that $x_\alpha\uoconv 0$ in $E$ implies that $x_\alpha\tauconv 0$.
Since order convergence implies unbounded order convergence, a \uoLtop\ is an \oLtop. A vector lattice $E$ need not admit a Hausdorff \uoLtop, but when it does, then this topology is unique (see \cite[Propositions~3.2,~3.4, and~6.2]{conradie:2005} or \cite[Theorems~5.5 and~5.9]{taylor:2019}) and we denote it by $\uoLt_E$.

Let $E$ be a vector lattice, let $F$ be an ideal of $E$, and suppose that $\tau_F$ is a (not necessarily Hausdorff) locally solid linear topology on $F$. Take a non-empty subset $S$ of $F$. Then there exists a unique (possibly non-Hausdorff) locally solid linear topology $\unb_S\tau_F$ on $E$ such that, for a net $\net$ in $E$, $x_\alpha\conv{\unb_S\tau_F}0$ if and only if $\abs{x_\alpha}\wedge\abs{s}\conv{\tau_F}0$ for all $s\in S$; see \cite[Theorem~3.1]{deng_de_jeu_UNPUBLISHED:2020a} for this, which extends earlier results in this vein in, e.g., \cite{conradie:2005} and \cite{taylor:2019}. This topology $\unb_S\tau_F$ is called the unbounded topology on $E$ that is generated by $\tau_F$ via $S$. Suppose that $E$ admits a Hausdorff \uoLtop\ $\uoLt_E$. The uniqueness of such a topology then implies that $\unb_E\uoLt_E=\uoLt_E$. In the sequel we shall use this result from \cite{conradie:2005} and \cite{taylor:2019} a few times.

Finally, the characteristic function of a set $S$ will be denoted by $\chi_S$, and the identity operator on a vector space will be denoted by $\Id$.

\section{Absolute strong operator topologies on $\regops(E,F)$}\label{sec:topologies_on_vector_lattices_of_order_bounded_operators}

\noindent Let $E$ and $F$ be vector lattices, where $F$ is Dedekind complete. In this section, we start by showing how topologies can be introduced on vector sublattices of $\regops(E,F)$ that can be regarded as absolute strong operator topologies; see \cref{res:topology_on_order_bounded_operators_practical} and \cref{rem:ASOT_remark}, below. Once this is known to be possible, it is easy to relate this to \oLtops\ and \uoLtops\ on regular vector sublattices of $\regops(E,F)$. In particular, we shall see that every regular vector sublattice of $\regops(E,F)$ admits a \necun\ Hausdorff \uoLtop\ when $F$ admits a Hausdorff \oLtop; see \cref{res:uoLtop_on_regular_sublattices_of_the_order_bounded_operators}, below.

When restricted to the orthomorphisms on a Dedekind complete vector lattice, the picture simplifies; see  \cref{sec:topologies_on_the_orthomorphisms_of_a_vector_lattice}. In particular, the restrictions of absolute strong operator topologies are then simply strong operator topologies.

The construction in the proof of the following result is an adaptation of that in the proof of \cite[Theorem~3.1]{deng_de_jeu_UNPUBLISHED:2020a}. The latter construction is carried out under minimal hypotheses and uses neighbourhood bases at zero as in \cite[proof of Theorem~2.3]{taylor:2019} rather than Riesz pseudo-norms. Such an approach enables one to also understand various `pathologies' in the literature from one central result; see \cite[Example~3.10]{deng_de_jeu_UNPUBLISHED:2020a}. It is for this reason of maximum flexibility that we also choose such a neighbourhood approach here.

\begin{theorem}\label{res:topology_on_order_bounded_operators}
	Let $E$ and $F$ be vector lattices, where $F$ is Dedekind complete, and let $\tau_F$ be a \uppars{not necessarily Hausdorff} locally solid additive topology on $F$. Take a non-empty subset $S$ of $E$.
	There exists a unique \uppars{possibly non-Hausdorff} additive topology $\optop{S}\tau_F$ on $\regops(E,F)$ such that, for a net $\opnet$ in $\regops(E,F)$, $T_\alpha\conv{\optop{S}\tau_F}0$  if and only if $\abs{T_\alpha}\abs{s}\conv{\tau_F}0$ for all $s\in S$.
	
	Let $I_S$ be the ideal of $E$ that is generated by $S$. For a net $\opnet$ in $\regops(E,F)$, $T_\alpha\conv{\optop{S}\tau_F}0$ if and only if $\abs{T_\alpha}\abs{x}\conv{\tau_F}0$ for all $x\in I_S$; and also if and only if  $\abs{T_\alpha}x\conv{\tau_F}0$ for all $x\in I_S$.

Furthermore:
\begin{enumerate}
	\item for every $x\in I_S$, the map $T\mapsto Tx$ is an $\optop{S}\tau_F\text{\textendash}\tau_F$ continuous map from $\regops(E,F)$ into $F$;\label{part:topology_on_operators_1}
	\item the topology $\optop{S}\tau_F$ on $\regops(E,F)$ is a locally solid additive topology;\label{part:topology_on_operators_2}
	\item \label{part:topology_on_operators_3} when $\tau_F$ is a Hausdorff topology on $F$, the following are equivalent for an additive subgroup $\opgroup$ of  $\regops(E,F)$:
	\begin{enumerate}
		\item the restriction $\optop{S}\tau_F|_\opgroup$ of $\optop{S}\tau_F$ to $\opgroup$ is a Hausdorff topology on $\opgroup$;\label{part:topology_on_operators_3a}
		\item $I_S$ separates the points of $\opgroup$.\label{part:topology_on_operators_3b}
	\end{enumerate}
\item\label{part:topology_on_operators_last} the following are equivalent for a linear subspace $\oplin$ of $\regops(E,F)$:
\begin{enumerate}
	\item for all $T\in\oplin$ and $s\in S$, $\abs{\varepsilon T}\abs{s}\conv{\tau_F}0$  as $\varepsilon\to 0$ in $\RR$;\label{part:linear_topology_on_operators_1}
	\item the restriction $\optop{S}\tau_F|_\oplin$ of $\optop{S}\tau_F$ to $\oplin$ is a \uppars{possibly non-Hausdorff} linear topology on $\oplin$.\label{part:linear_topology_on_operators_last}
\end{enumerate}

\end{enumerate}	
\end{theorem}

\begin{proof}
Suppose that $\tau_F$ is a (not necessarily Hausdorff) locally solid additive topology on $F$.

It is clear from the required translation invariance of $\optop{S}\tau_F$ that it is unique, since the nets that are $\optop{S}\tau_F$-convergent to zero are prescribed.

For its existence, we take a $\tau_F$-neighbourhood base $\{U_\lambda\}_{\lambda\in\Lambda}$ of zero in $F$ that consists of solid subsets of $F$. For $x\in I_S$ and $\lambda\in\Lambda$, we set
\[
V_{\lambda,x}\coloneqq \{\, T\in\regops(E,F): \abs{T}\abs{x}\in U_\lambda\,\}.
\]
The $V_{\lambda,x}$'s are solid subsets of $\regops(E,F)$ since the $U_\lambda$ are solid subsets of $F$.

Set
\[
\mathcal N_0\coloneqq \{\,V_{\lambda,x}: \lambda\in\Lambda,x \in I_S \,\}.
\]

We shall now verify that $\mathcal N_0$ satisfies the necessary and sufficient conditions in \cite[Theorem~3 on p.~46]{husain_INTRODUCTION-TO_TOPOLOGICAL_GROUPS:1966} to be a base of neighbourhoods of zero for an additive topology on $\regops(E,F)$.

Take $V_{\lambda_1,x_1}, V_{\lambda_2,x_2}\in\mathcal N_0$. There exists a $\lambda_3\in\Lambda$ such that $U_{\lambda_3}\subseteq U_{\lambda_1}\cap U_{\lambda_2}$, and it is easy to verify that then $V_{\lambda_3,\abs{x_1}\vee\abs{x_2}}\subseteq V_{\lambda_1,x_1}\cap V_{\lambda_2,x_2}$. Hence $\mathcal N_0$ is a filter base.

It is clear that $V_{\lambda,x}=- V_{\lambda,x}$.

Take $V_{\lambda,x}\in\mathcal N_0$. There exists a $\mu\in\Lambda$ such that $U_\mu+U_\mu\subseteq U_\lambda$, and it is easy to see that then $V_{\mu,x}+V_{\mu,x}\subseteq V_{\lambda,x}$.

 An appeal to \cite[Theorem~3 on p.~46]{husain_INTRODUCTION-TO_TOPOLOGICAL_GROUPS:1966} now yields that $\mathcal N_0$ is a base of neighbourhoods of zero for an additive topology on $\regops(E,F)$ that we shall denote by $\optop{S}\tau_F$. It is a direct consequence of its definition that, for a net  $\opnet$ in $\regops(E,F)$, $T_\alpha\conv{\optop{S}\tau_F}0$ if and only if $\abs{T_\alpha}\abs{x}\conv{\tau_F}0$ for all $x\in I_S$. Using the fact that $\tau_F$ is a locally solid additive topology on $F$, it is routine to verify that the latter condition is equivalent to the condition that  $\abs{T}x\conv{\tau_F} 0$ for all $x\in I_S$, as well as to the condition that $\abs{T_\alpha}\abs{s}\conv{\tau_F}0$ for all $s\in S$.

 We turn to the statements in the parts~\partref{part:topology_on_operators_1}--\partref{part:topology_on_operators_last}.

 For part~\partref{part:topology_on_operators_1}, suppose that $\opnet$ is a net in $\regops(E,F)$ such that $T_\alpha\conv{\optop{S}\tau_F}0$. Then $\abs{T_\alpha}\abs{x}\conv{\tau_F} 0$ for all $x\in I_S$. Since $\abs{T_\alpha x}\leq  \abs{T_\alpha}\abs{x}$, the fact that $\tau_F$ is locally solid implies that then also $T_\alpha x\conv{\tau_F}0$ for all $x\in I_S$.

Since the topology $\optop{S} \tau_F$ is a locally solid additive topology on $\regops(E,F)$ by construction, part~\partref{part:topology_on_operators_2} is clear.

 For part~\partref{part:topology_on_operators_3}, we recall from \cite[p.~48, Theorem~4]{husain_INTRODUCTION-TO_TOPOLOGICAL_GROUPS:1966} that an additive topology on a group is Hausdorff if and only if the intersection of the elements of a neighbourhood base of zero is trivial. Using this for $F$ in the second step, and invoking \cite[Proposition~2.1]{deng_de_jeu_UNPUBLISHED:2020a} in the third, we see that
 \begin{align*}
 \bigcap_{\lambda\in\Lambda,x\in I_S} \left(V_{\lambda,x}\cap \opgroup\right)&=\{\,T\in\regops(E,F):\abs{T}\abs{x}\in\bigcap_{\lambda\in\Lambda}U_\lambda\text{ for all }x\in I_S\,\}\cap \opgroup\\
&= \{\,T\in\regops(E,F):\abs{T}\abs{x}=0\text{ for all }x\in I_S\,\}\cap \opgroup.\\
&=\{\,T\in\regops(E,F): Tx=0\text{ for all }x\in I_S\,\}\cap\opgroup\\
&=\{\,T\in\opgroup: Tx=0\text{ for all }x\in I_S\,\}.
 \end{align*}
Another appeal to \cite[p.~48, Theorem~4]{husain_INTRODUCTION-TO_TOPOLOGICAL_GROUPS:1966} then completes the proof of part~\partref{part:topology_on_operators_3}.

We prove that part~\partref{part:linear_topology_on_operators_1} implies part~ \partref{part:linear_topology_on_operators_last}. It is clear that $\optop{S}\tau_F|_\oplin$ is an additive topology on $\oplin$.
From what we have already established, we know that the assumption implies that also $\abs{\varepsilon T}\abs{x}\conv{\tau_F}0$ as $\varepsilon\to 0$ in $\RR$ for all $T\in\oplin$ and $x\in I_S$.
Fix $\lambda\in\Lambda$ and $x\in I_S$, and take $T\in\oplin$. Since $\abs{\varepsilon T}\abs{x}\conv{\tau_F}0$  as $\varepsilon\to 0$ in $\RR$, there exists a $\delta>0$ such that $\abs{\varepsilon T}\abs{x}\in U_{\lambda}$ whenever $\abs{\varepsilon}<\delta$. That is, $\varepsilon T\in V_{\lambda,x}\cap\oplin$ whenever $\abs{\varepsilon}<\delta$. Hence $V_{\lambda,x}\cap \oplin$ is an  absorbing subset of $\oplin$. Furthermore, since $V_{\lambda,x}$ is a solid subset of $\regops(E,F)$, it is clear that $\varepsilon T\in V_{\lambda,x}\cap\oplin$ whenever $T\in V_{\lambda,x}\cap \oplin$ and $\varepsilon\in[-1,1]$. We conclude from \cite[Theorem~5.6]{aliprantis_border_INFINITE_DIMENSIONAL_ANALYSIS_THIRD_EDITION:2006} that  $\optop{S}\tau_F|_\oplin$ is a linear topology on $\oplin$.

We prove that part~\partref{part:linear_topology_on_operators_last} implies part~\partref{part:linear_topology_on_operators_1}. Take $T\in\oplin$. Then $\varepsilon T\conv{\optop{S}\tau_F|_\oplin}0$ as $\varepsilon\to 0$ in $\RR$. By construction, this implies that (and is, in fact, equivalent to) the fact that $\abs{\varepsilon T}\abs{s}\conv{\tau_F}0$ for all $s\in S$.

\end{proof}	

\begin{remark}
	It is clear from the convergence criteria for nets that the topologies $\optop{S_1}\tau_F$ and $\optop{S_2}\tau_F$ are equal when $I_{S_1}=I_{S_2}$. One could, therefore, work with ideals from the very start, but it seems worthwhile to keep track of a smaller set of presumably more manageable `test vectors'. See also the comments preceding \cref{res:uoLtops_on_lattices_of_operators_most_precise}, below.
\end{remark}

\begin{remark}\label{rem:uniform_convergence_on_order_intervals}
Suppose that $\opnet$ is a net in $\regops(E,F)$ such that $T_\alpha\conv{\optop{S}\tau_F}0$. It is easy to see that then $\abs{T_\alpha} x\conv{\tau_F}0$ uniformly on every order bounded subset of $I_S$, so that then also $T_\alpha x\conv{\tau_F}0$ uniformly on every order bounded subset of $I_S$.
\end{remark}

\begin{definition}
	The topology $\optop{S}\tau_F$ in \cref{res:topology_on_order_bounded_operators}  is called the \emph{absolute strong operator topology that is generated by $\tau_F$ via $S$}. We shall comment on this nomenclature in \cref{rem:ASOT_remark}, below.
\end{definition}

The following result, which can also be obtained using Riesz pseudo-norms, is clear from \cref{res:topology_on_order_bounded_operators}.

\begin{corollary}\label{res:topology_on_order_bounded_operators_practical}
	Let $E$ and $F$ be vector lattices, where $F$ is Dedekind complete, and let $\tau_F$ be a \uppars{not necessarily Hausdorff} locally solid linear topology on $F$. Take a vector sublattice  $\oplat$ of $\regops(E,F)$ and a non-empty subset $S$ of $E$.
	
	There exists a unique additive topology $\optop{S}\tau_F$ on $\oplat$ such that, for a net $\opnet$ in $\oplat$, $T_\alpha\conv{\optop{S}\tau_F}0$ if and only if $\abs{T_\alpha}\abs{s}\conv{\tau_F}0$ for all $s\in S$.
	
	Let $I_S$ be the ideal of $E$ that is generated by $S$.
	For a net $\opnet$ in $\oplat$, $T_\alpha\!\conv{\optop{S}\tau_F}0$ if and only if $\abs{T_\alpha}\abs{x}\conv{\tau_F}0$ for all $x\in I_S$; and also if and only if  $\abs{T_\alpha}x\conv{\tau_F}0$ for all $x\in I_S$.
	
	Furthermore:
	\begin{enumerate}
		\item for every $x\in I_S$, the map $T\mapsto Tx$ is an $\optop{S}\tau_F\text{\textendash}\tau_F$ continuous map from $\oplat$ into $F$;\label{part:topology_on_operators_practial_1}
		\item the additive topology $\optop{S}\tau_F$ on the group $\oplat$ is, in fact, a locally solid linear  topology on the vector lattice $\oplat$.  When $\tau_F$ is a Hausdorff topology on $F$, then $\optop{S}\tau_F$ is a Hausdorff topology on $\oplat$ if and only if $I_S$ separates the points of $\oplat$.\label{part:topology_on_operators_practical_2}
	\end{enumerate}
\end{corollary}

\begin{remark}
	Although in the sequel of this paper we shall mainly be interested in the nets that are convergent in a given topology, let us still remark that is possible to describe an explicit $\optop{S}\tau_F$-neighbourhood base of zero in $\oplat$.
	Take a $\tau_F$-neighbourhood base $\{U_\lambda\}_{\lambda\in\Lambda}$ of zero in $F$ that consists of solid subsets of $F$. For $\lambda\in\Lambda$ and $x\in I_S$, set
	\[
	V_{\lambda,x}\coloneqq \{\, T\in\oplat: \abs{T}\abs{x}\in U_\lambda\,\}.
	\]
	Then $\{\,V_{\lambda,x}: \lambda\in\Lambda,x \in I_S \,\}$ is an $\optop{S}\tau_F$-neighbourhood base of zero in $\oplat$.
\end{remark}

\begin{remark}\label{rem:ASOT_remark}
	It is not difficult to see that $\optop{S}\tau_F$ is the weakest \emph{locally solid} linear topology $\tau_\oplat$ on $\oplat$ such that, for every  $x\in I_S$, the map $T\to Tx$ is a $\tau_{\oplat}\text{\textendash}\tau_F$ continuous map from $\oplat$ into $F$. It is also the weakest linear topology $\tau_\oplat^\prime$ on $\oplat$ such that, for every  $x\in I_S$, the map $T\to \abs{T}x$ is a $\tau_\oplat^\prime \text{\textendash}\tau_F$ continuous map from $\oplat$ into $F$. The latter characterisation is our motivation for the name `absolute strong operator topology'.
	
	Take $F=\RR$ and $S=E$. Then $\optop{E}\tau_\RR$ is what is commonly known as the absolute weak$^\ast$-topology on $\odual{E}$. There is an unfortunate clash of `weak' and `strong' here that appears to be unavoidable.
\end{remark}

\begin{remark}
	For comparison with \cref{rem:ASOT_remark}, and to make clear the role of the local solidness of the topologies in the present section, we mention the following, which is an easy consequence of \cite[Theorem~5.6]{aliprantis_border_INFINITE_DIMENSIONAL_ANALYSIS_THIRD_EDITION:2006}, for example. Let $E$ and $F$ be vector spaces, where $F$ is supplied with a (not necessarily) Hausdorff linear topology $\tau_F$. Take a linear subspace $\oplat$ of the vector space of all linear maps from $E$ into $F$, and take a non-empty subset $S$ of $E$. Then there exists a unique (not necessarily Hausdorff) linear topology $\mathrm{SOT}_S\tau_F$ on $\mathcal E$ such that, for a net $\opnet$ in $\oplat$, $T_\alpha\conv {\mathrm{SOT}_S\tau_F} 0$ if and only if $T_\alpha s\conv{\tau_F} 0$ for all $s\in S$. The subsets of $\oplat$ of the form $\bigcap_{i=1}^n\{\,T\in\oplat : Ts_i\in V_{\lambda_i}\,\}$, where the $s_i$ run over $S$ and the $V_{\lambda_i}$ run over a balanced $\tau_F$-neighbourhood base $\{\,V_\lambda: \lambda\in\Lambda\,\}$ of zero in $F$, are an $\mathrm{SOT}_S\tau_F$-neighbourhood base of zero in $\oplat$. When $\tau_F$ is Hausdorff, then $\mathrm{SOT}_S\tau_F$ is Hausdorff if and only if $S$ separates the points of $\oplat$.  This strong operator topology $\mathrm{SOT}_S\tau_F$ on $\oplat$ that is generated by $\tau_F$ via $S$, is the weakest linear topology $\tau_\oplat$ on $\oplat$ such that, for every $s\in S$, the map $T\mapsto Tx$ is $\tau_\oplat\text{\textendash}\tau_F$-continuous.
\end{remark}

\section{\oLtops\ and \uoLtops\ on vector lattices of operators}\label{sec:lebesgue_topologies_on_vector_lattices_of_operators}

\noindent
To arrive at results concerning \oLtops\ and \uoLtops\ on regular vector sublattices of operators, we need a preparatory result for which we are not aware of a reference. Given its elementary nature, we refrain from any claim to originality. It will re-appear at several places in the sequel.

\begin{lemma}\label{res:o_to_o}
	Let $E$ and $F$ be vector lattices, where $F$ is Dedekind complete, and let $\oplat$ be a regular vector sublattice of $\regops(E,F)$. Suppose that $\opnet$ is net in $\oplat$ such that $T_\alpha\oconv 0$ in $\oplat$. Then $T_\alpha x\oconv 0$ for all $x\in E$.
\end{lemma}

\begin{proof}
	By the regularity of $\oplat$, we also have that $T_\alpha\oconv 0$ in $\obops(E,F)$. Hence there exists a net $\netgen{S_\beta}{\beta\in\istwo}$ in $\regops(E,F)$ such that $S_\beta\downarrow 0$ in $\regops(E,F)$ and with the property that, for every $\beta_0\in\istwo$, there exists an $\alpha_0\in\is$ such that $\abs{T_\alpha}\leq S_{\beta_0}$ for all $\alpha\in\is$ such that $\alpha\geq\alpha_0$. We know from \cite[Theorem 1.18]{aliprantis_burkinshaw_POSITIVE_OPERATORS_SPRINGER_REPRINT:2006}, for example, that $S_\beta x\downarrow 0$ for all $x\in \posE$. Since $\abs{T_\alpha x}\leq \abs{T_\alpha} x$ for $x\in\posE$, it then follows easily that $T_\alpha x\oconv 0$ for all $x\in\posE$. Hence $T_\alpha x\oconv 0$ for all $x\in E$.
\end{proof}

We can now show that the \oL\ property of a locally solid linear topology on the Dedekind complete codomain is inherited by the associated absolute strong operator topology on a regular vector sublattice of operators.

\begin{proposition}\label{res:oLtop_on_vector_lattice_of_operators}
	Let $E$ and $F$ be vector lattices, where $F$ is Dedekind complete. Suppose that $F$ admits an \oLtop\ $\tau_F$. Take a regular vector sublattice  $\oplat$ of $\regops(E,F)$ and a non-empty subset $S$ of $E$.
	Then $\optop{S}\tau_F$ is an \oLtop\ on $\oplat$. When $\tau_F$ is a Hausdorff topology on $F$, then $\optop{S}\tau_F$ is a Hausdorff topology on $\oplat$ if and only if $I_S$ separates the points of $\oplat$.
\end{proposition}

\begin{proof}
	In view of \cref{res:topology_on_order_bounded_operators_practical}, we merely need to show that, for a net $\opnet$ in $\oplat$, the fact that $T_\alpha\oconv 0$ in $\oplat$ implies that $T_\alpha\conv{\optop{S}\tau_F}0$. Take $s\in S$. Since also $\abs{T_\alpha}\oconv 0$ in $\oplat$, \cref{res:o_to_o} implies that $\abs{T_\alpha}\abs{s}\oconv 0$ in $F$. Using that $\tau_F$ is an \oLtop\ on $F$, we find that $\abs{T_\alpha}\abs{s}\conv{\tau_F} 0$. Since this holds for all $s\in S$, \cref{res:topology_on_order_bounded_operators_practical} shows that $T_\alpha\conv{\optop{S}\tau_F}0$ in $\oplat$.
	\end{proof}

We conclude by showing that every regular vector sublattice of $\regops(E,F)$ admits a \necun\ Hausdorff \uoLtop\ when the Dedekind complete codomain $F$ admits a Hausdorff \oLtop. It is the unbounded topology that is associated with the members of a family of absolute strong operator topologies on the vector sublattice, with all members yielding the same result. Our most precise result in this direction is the following. The convergence criterion in part~(2) is a `minimal one' that is convenient when one wants to show that a net is convergent, whereas the criterion in part~(3) maximally exploit the known convergence of a net.

\begin{theorem}\label{res:uoLtops_on_lattices_of_operators_most_precise}
	Let $E$ and $F$ be vector lattices, where $F$ is Dedekind complete. Suppose that $F$ admits an \oLtop\ $\tau_F$. Take a regular vector sublattice  $\oplat$ of $\regops(E,F)$, a non-empty subset $\opset$ of $\oplat$, and a non-empty subset $S$ of $E$.

Then $\unb_\opset\optop{S}\tau_F$ is a \uoLtop\ on $\oplat$.

We let $I_S$ denote the ideal of $E$ that is generated by $S$, and $I_\opset$ the ideal of $\oplat$ that is generated by $\opset$. For a net $\opnet$ in $\oplat$, the following are equivalent:
\begin{enumerate}
	\item $T_\alpha\conv{\unb_\opset\optop{S}\tau_F}0$;
	\item $(\abs{T_\alpha}\wedge \abs{T})\abs{s}\conv{\tau_F} 0$ for all $T\in\opset$ and $s\in S$;
	\item $(\abs{T_\alpha}\wedge \abs{T})x\conv{\tau_F} 0$ for all $T\in I_\opset$ and $x\in I_S$.
\end{enumerate}

Suppose that $\tau_F$ is actually a \emph{Hausdorff} \oLtop\ on $F$. Then the following are equivalent:
\begin{enumerate}
	\item[(i)] $\unb_\opset\optop{S}\tau_F$ is a \necun\ \emph{Hausdorff} \uoLtop\ on~$\oplat$;
	\item[(ii)] $I_S$ separates the points of $\oplat$ and $I_\opset$ is order dense in $\oplat$.
\end{enumerate}

In that case, the Hausdorff \uoLtop\ $\unb_\opset\optop{S}\tau_F$ on $\oplat$ is the restriction of the \necun\ Hausdorff \uoLtop\ on $\regops(E,F)$, i.e., of $\unb_{\regops(E,F)}\optop{E}\tau_F$, and the criteria in~\uppars{1},~\uppars{2}, and~\uppars{3} are also equivalent to:
\begin{enumerate}
\item[(4)] $(\abs{T_\alpha}\wedge \abs{T})x\conv{\tau_F} 0$ for all $T\in\regops(E,F)$ and $x\in E$.
\end{enumerate}
\end{theorem}

\begin{proof}
	\!\! It is clear from \cref{res:oLtop_on_vector_lattice_of_operators} and \cite[Proposition~4.1]{deng_de_jeu_UNPUBLISHED:2020a} that $\unb_\opset\optop{S}\tau_F$ is a \uoLtop\ on $\oplat$. The two convergence criteria for nets follow from the combination of those in \cite[Theorem~3.1]{deng_de_jeu_UNPUBLISHED:2020a} and in \cref{res:topology_on_order_bounded_operators_practical}.
	
	According to \cite[Proposition~4.1]{deng_de_jeu_UNPUBLISHED:2020a}, $\unb_\opset\optop{S}\tau_F$ is a Hausdorff topology on $\oplat$ if and only if $\optop{S}\tau_F$ is a Hausdorff topology on $\oplat$ and $I_\opset$ is order dense in $\oplat$. An appeal to \cref{res:oLtop_on_vector_lattice_of_operators} then completes the proof of the necessary and sufficient conditions for $\unb_\opset\optop{S}\tau_F$ to be Hausdorff.
	
	Suppose that $\tau_F$ is actually also Hausdorff, that $I_S$ separates the points of $\oplat$, and that $I_\opset$ is order dense in $\oplat$. From what we have already established, it is clear that $\unb_{\regops(E,F)}\optop{E}\tau_F$ is a \necun\ Hausdorff \uoLtop\ on $\regops(E,F)$. Since the restriction of a Hausdorff \uoLtop\ on a vector lattice to a regular vector sublattice is a \necun\ Hausdorff \uoLtop\ on the vector sublattice (see \cite[Proposition~5.12]{taylor:2019}), the criterion in part~(4) follows from that in part~(3) applied to $\unb_{\regops(E,F)}\optop{E}\tau_F$.
\end{proof}

\begin{remark}
	Take a $\tau_F$-neighbourhood base $\{U_\lambda\}_{\lambda\in\Lambda}$ of zero in $F$ that consists of solid subsets of $F$. For $\lambda\in\Lambda$, $\widetilde T\in I_\opset$, and $x\in I_S$, set
\[
V_{\lambda,\widetilde T,x}\coloneqq \{\, T\in\oplat: (\abs{T}\wedge \abs{\widetilde T})\abs{x}\in U_\lambda\,\}.
\]
As a consequence of the constructions of unbounded and absolute strong operator topologies, $\{\,V_{\lambda,\widetilde{T},x}: \lambda\in\Lambda, T\in I_\opset, x \in I_S \,\}$ is then a $\unb_\opset\optop{S}\tau_F$-neighbourhood base of zero in $\oplat$.
\end{remark}

The following consequence of \cref{res:uoLtops_on_lattices_of_operators_most_precise} will be sufficient in many situations.

\begin{corollary}\label{res:uoLtop_on_regular_sublattices_of_the_order_bounded_operators}
	Let $E$ and $F$ be vector lattices, where $F$ is Dedekind complete. Suppose that $F$ admits a Hausdorff \oLtop\ $\tau_F$.
	
	Take a regular vector sublattice $\oplat$ of $\regops(E,F)$. Then $\oplat$ admits a \necun\ Hausdorff \uoLtop\ $\uoLt_\oplat$. This topology equals $\unb_\oplat\optop{E}\tau_F$, and is also equal to the restriction to $\oplat$ of the Hausdorff \uoLtop\ $\unb_{\regops(E,F)}\optop{E}\tau_F$ on $\regops(E,F)$.

For a net $\opnet$ in $\oplat$, the following are equivalent:
\begin{enumerate}
	\item $T_\alpha\conv{\uoLt_\oplat}0$;
	\item $(\abs{T_\alpha}\wedge \abs{T}){x}\conv{\tau_F} 0$ for all $T\in\oplat$ and $x\in E$;	
	\item $(\abs{T_\alpha}\wedge \abs{T}){x}\conv{\tau_F} 0$ for all $T\in\regops(E,F)$ and $x\in E$.	
\end{enumerate}
\end{corollary}

\begin{remark}
	There can, sometimes, be other ways to see that a given regular vector sublattice of $\regops(E,F)$ admits a Hausdorff \uoLtop. For example, suppose that $\ocdual{F}$ separates the points of $F$. For $x\in E$ and $\varphi\in\ocdual{F}$, the map $T\mapsto\varphi(Tx)$ defines an order continuous linear functional on $\ocops(E,F)$, and it is then clear that the order continuous dual of $\ocops(E,F)$ separates the points of $\ocops(E,F)$. Hence $\ocops(E,F)$ can also be supplied with a Hausdorff \uoLtop\ as in \cite[Theorem~5.2]{deng_de_jeu_UNPUBLISHED:2020a} which, in view of its uniqueness, coincides with the one as supplied by \cref{res:uoLtop_on_regular_sublattices_of_the_order_bounded_operators}.
\end{remark}

\section{Comparing uniform and strong convergence structures on $\regops(E,F)$}\label{sec:uniform_and_strong_convergence_structures_on_regops}

\noindent
Suppose that $E$ and $F$ are vector lattices, where $F$ is Dedekind complete. As explained in \cref{sec:introduction_and_overview}, there exist a uniform and a strong convergence structure on $\regops(E,F)$ for each of order convergence, unbounded order convergence, and\textemdash when applicable\textemdash convergence in the Hausdorff \uoLtop. In this section, we investigate what the inclusion relations are between the members of each of these three pairs. For example, is it true that the uniform (resp.\ strong) order convergence of a net of order bounded operators implies its strong (resp.\ uniform) order convergence to the same limit? We shall show that only one of the six conceivable implications is valid in general, and that the others are not even generally valid for uniformly bounded sequences of order continuous operators on Banach lattices.
Whilst the failures of such general implications may, perhaps, not come as too big a surprise, the positive results for orthomorphisms (see \cref{res:ordertoorderinorth,res:oRKF,res:uoRKF,res:uoLt_RKF}, below) may serve to indicate that they are less evident than one would think at first sight.

For monotone nets in $\regops(E,F)$, however, the following result shows that then even all four (or six) notions of convergence in $\regops(E,F)$ coincide.

\begin{proposition}\label{res:regops_monotone_nets}
Let $E$ and $F$ be vector lattices, where $F$ is Dedekind complete, and let $\opnet$ be a monotone net in $\regops(E,F)$. The following are equivalent:
\begin{enumerate}
	\item $T_\alpha\oconv 0$ in $\regops(E,F)$;
	\item $T_\alpha\uoconv 0$ in $\regops(E,F)$;
	\item $T_\alpha x\oconv 0$ in $F$ for all $x\in E$;
	\item $T_\alpha x \uoconv 0$ in $F$ for all $x\in E$.
	\end{enumerate}
Suppose that, in addition, $F$ admits a \necun\ Hausdorff \uoLtop\ $\uoLt_F$, so that $\regops(E,F)$ also admits a \necun\ Hausdorff \uoLtop\ $\uoLt_{\regops(E,F)}$ by \cref{res:uoLtop_on_regular_sublattices_of_the_order_bounded_operators}. Then \uppars{1}--\uppars{4} are also equivalent to:
\begin{enumerate}
	\setcounter{enumi}{4}
	\item $T_\alpha\conv{\uoLt_{\regops(E,F)}} 0$;
	\item $T_\alpha x\conv{\uoLt_F} 0$ for all $x\in E$.
\end{enumerate}
\end{proposition}

\begin{proof}
We may suppose that $T_\alpha\downarrow$ and that $x\in\pos{E}$. For order bounded nets in a vector lattice, order convergence and unbounded order convergence are equivalent. Passing to an order bounded tail of $\opnet$, we thus see that the parts~(1) and~(2) are equivalent. Similarly, the parts~(3) and~(4) are equivalent. The equivalence of the parts~(1) and~(3) is well known; see \cite[Theorem~1.67]{aliprantis_burkinshaw_LOCALLY_SOLID_RIESZ_SPACES_WITH_APPLICATIONS_TO_ECONOMICS_SECOND_EDITION:2003}, for example.

Suppose that $F$ admits a Hausdorff \uoLtop\ $\uoLt_F$. In that case, it follows from \cite[Lemma~7.2]{deng_de_jeu_UNPUBLISHED:2020a} that the parts~(2) and~(5) are equivalent, as are the parts~(4) and~(6).
\end{proof}

When $\opnet$ is a not necessarily monotone net in $\regops(E,F)$ such that $T_\alpha\oconv 0$, then \cref{res:o_to_o} shows that $T_\alpha x\oconv 0$ in $F$ for all $x\in E$. We shall now give five examples to show that each of the remaining five conceivable implications between a corresponding uniform and strong convergence structures on $\obops(E,F)$ is not generally valid. In each of these examples, we can even take $E=F$ to be a Banach lattice, and for the net $\opnet$ we can even take a uniformly bounded sequence $\seq{T_n}{n}$ of order continuous operators on $E$.

\begin{example}\label{exam:counterexample_1}
	We give an example of a \emph{uniformly bounded sequence $\seq{T_n}{n}$ of positive order continuous operators on a Dedekind complete Banach lattice $E$ with a strong order unit, such that $T_n x\oconv 0$ in $E$ for all $x\in E$ but $T_n\xslashedrightarrowc{\o} 0$ in $\regops(E)$} because the sequence is not even order bounded in $\obops(E)$.
	
	We choose $\ell_{\infty}(\NN)$ for $E=F$. For $n\geq 1$, we set $T_n\coloneqq S^n$, where $S$ is the right shift operator on $E$. The $T_n$ are evidently positive and of norm one. A moment's thought shows that they are order continuous. Furthermore, it is easy to see that $T_n x\oconv 0$ in $E$ for all $x\in E$. We shall now show that $\{\,T_n: n\geq 1\,\}$ is not order bounded in $\regops(E)$. For this, we start by establishing that the $T_n$ are mutually disjoint. Let $\seq{e_i}{i}$ be the standard sequence of unit vectors in $E$. Take $m\neq n$ and $i\geq 1$. Since $e_i$ is an atom, the Riesz--Kantorovich formula for the infimum of two operators shows that
	\[
	0\leq (T_m\wedge T_n)e_i=\inf\{\,te_{m+i}+(1-t)e_{n+i}: 0\leq t\leq 1\,\}\leq\inf\{e_{m+i},e_{n+i}\}=0.
	\]
	Hence $(T_m\wedge T_n)$ vanishes on the span of the $e_i$. Since this span is order dense in $E$, and since $T_n\wedge T_m\in \ocops(E)$, it follows that $T_n\wedge T_m=0$.
	
	We can now show that $\seq{T_n}{n}$ is not order bounded in $\regops(E)$. Indeed, suppose that $T\in\regops(E)$ is a upper bound for all $T_n$. Set $e\coloneqq\bigvee_{i=1}^{\infty}e_i$. Then, for all $N\geq 1$,
	\[
	Te\geq \left(\bigvee_{n=1}^{N}T_n\right)e=\left(\sum_{n=1}^{N}T_n\right)e\geq Ne_{N+1}.
	\]
This shows that $Te$ cannot be an element of $\ell_\infty$. We conclude from this contradiction that $\seq{T_n}{n}$ is not order bounded in $\regops(E)$.
\end{example}

\begin{example}\label{exam:counterexample_2}
We give an example of a \emph{uniformly bounded sequence $\seq{T_n}{n}$ of positive order continuous operators on a Dedekind complete Banach lattice $E$ with a strong order unit, such that $T_n\uoconv 0$ in $\regops(E)$ but $T_n x\xslashedrightarrowc{\uo} 0$ for some $x\in E$}.

We choose $\ell_\infty(\ZZ)$ for $E=F$. For $n\geq 1$, we set $T_n\coloneqq S^n$, where $S$ is the right shift operator on $E$. Just as in \cref{exam:counterexample_1}, the $T_n$ are  positive order continuous operators on $E$ of norm one that are mutually disjoint. Since disjoint sequences in vector lattices are unbounded order convergent to zero (see \cite[Corollary~3.6]{gao_troitsky_xanthos:2017}), we have $T_n\uoconv 0$ in $\regops(E)$. On the other hand, if we let $e$ be the two-sided sequence that is constant 1, then $T_ne=e$ for all $n\geq 1$. Hence $\seq{T_ne}{n}$ is not unbounded order convergent to zero in $E$.
	\end{example}

For our next example, we require a preparatory lemma.

\begin{lemma}\label{res:disjointness_conditional_expectation_and_identity}
	Let $\mu$ be the Lebesgue measure on the Borel $\sigma$-algebra $\mathcal B$ of $[0,1]$, and let $1\leq p\leq\infty$. Take a Borel subset $S$ of $[0,1]$, and define the positive operator  $T_S:\Ell_p([0,1],\mathcal B,\mu)\to\Ell_p([0,1],\mathcal B,\mu)$ by setting
	\[
	T_S(f)\coloneqq \int_S\!\! f\,\mathrm{d}\mu\cdot \chi_S
	\]
	for $f\in \Ell_p([0,1],\mathcal B,\mu)$. Then $T_S\wedge I=0$.
\end{lemma}

\begin{proof}
	
	Take an $n\geq 1$, and choose disjoint a partition $[0,1]=\bigcup_{i=1}^n A_i$ of $[0,1]$ into Borel sets $A_i$ of measure $1/n$. Let $e$ denote the constant function 1. Then
	\begin{align*}
	(T_S\wedge I)e&=\sum_{i=1}^n (T_S\wedge I)\chi_{A_i}\\
	&\leq\sum_{i=1}^n (T_S\chi_{A_i})\wedge \chi_{A_i}\\
	&\leq\sum_{i=1}^n(\mu(A_i)\chi_S)\wedge\chi_{A_i}\\
	&\leq\sum_{i=1}^n \mu(A_i)\chi_{A_i}\\
	&=\frac{1}{n} e.
	\end{align*}
	Since $n$ is arbitrary, we see that $(T_S\wedge I)e=0$. Because $0\leq T_S\wedge \Id\leq\Id$, $T_S\wedge \Id$ is order continuous. From the fact that the positive order continuous operator $T_S\wedge I$ vanishes on the weak order unit $e$ of $\Ell_p([0,1],\mathcal B,\mu)$, we conclude that $T_S\wedge I=0$.
\end{proof}

\begin{example}\label{exam:counterexample_4}
	We give an example of a \emph{uniformly bounded sequence $\seq{T_n}{n}$ of order continuous operators on a separable reflexive Banach lattice $E$ with a weak order unit, such that $T_nx\uoconv 0$ in $E$ for all $x\in E$ but $T_n\xslashedrightarrowc{\uo}0$ in $\regops(E)$} because even $T_n\xslashedrightarrowc{\uoLt_{\regops(E)}}0$ in $\regops(E)$.
	
	Let $\mu$ be the Lebesgue measure on the Borel $\sigma$-algebra $\mathcal B$ of $[0,1]$, and let $1\leq p\leq \infty$. For $E$ we choose $\Ell_p([0,1],\mathcal B,\mu)$, so that $E$ is reflexive for $1<p<\infty$. For $n\geq 1$, we let $\mathcal B_n$ be the sub-$\sigma$-algebra of $\mathcal B$ that is generated by the intervals $S_{n,i}\coloneqq [(i-1)/2^n,i/2^n]$ for $i=1,\dotsc,2^n$, and we let $\EE_n:E\to E$ be the corresponding conditional expectation. By \cite[Theorem~10.1.5]{bogachev_MEASURE_THEORY_VOLUME_II:2007}, $\EE_n$ is a positive norm one projection. A moment's thought shows that every open subset of $[0,1]$ is the union of the countably infinitely many $S_{n,i}$ that are contained in it, so that it follows from \cite[Theorem~10.2.3]{bogachev_MEASURE_THEORY_VOLUME_II:2007} that $\EE_n f\to f$ almost everywhere as $n\to\infty$. By \cite[Proposition~3.1]{gao_troitsky_xanthos:2017}, we can now conclude that $\EE_n f\uoconv f$ for all $f\in E$.
	
	On the other hand, it is not true that $\EE_n\conv{\uoLt_{\regops(E)}} I$. To see this, we note that, by \cite[Example~10.1.2]{bogachev_MEASURE_THEORY_VOLUME_II:2007}, every $\EE_n$ is a linear combination of operators as in \cref{res:disjointness_conditional_expectation_and_identity}. Hence $\EE_n\perp I$ for all $n$. Since $\uoLt_{\regops(E)}$ is a locally solid linear topology, a possible $\uoLt_{\regops(E)}$-limit of the $\EE_n$ is also disjoint from $I$, hence cannot be $I$ itself.
	
	On setting $T_n\coloneqq \EE_n-I$ for $n\geq 1$, we have obtained a sequence of operators as desired.
	\end{example}

\begin{example}\label{exam:counterexample_3}
	We give an example of a \emph{uniformly bounded sequence $\seq{T_n}{n}$ of positive order continuous operators on a Dedekind complete Banach lattice $E$ with a strong order unit that admits a Hausdorff \uoLtop, such that $T_n\conv{\uoLt_{\regops(E)}} 0$ in $\regops(E)$ but $T_n x\xslashedrightarrowc{\uoLt_E} 0$ in $E$ for some $x\in E$}.
	
	We choose $E$,  the $T_n\in\regops(E)$, and $e\in E$ as in \cref{exam:counterexample_2}. There are several ways to see that $E$ admits a Hausdorff \uoLtop. This follows most easily from the fact that $E$ is atomic (see \cite[Lemma~7.4]{taylor:2019}) and also from  \cite[Theorem~6.3]{deng_de_jeu_UNPUBLISHED:2020a} in the context of measure spaces. By \cref{res:uoLtop_on_regular_sublattices_of_the_order_bounded_operators}, $\regops(E)$ then also admits such a topology. Since we already know from \cref{exam:counterexample_2}  that $T_n\uoconv 0$, we also have that $T_n\conv{\uoLt_{\regops(E)}}0$. On the other hand, the fact that $T_ne=e$ for $n\geq 1$ evidently shows that $\seq{T_ne}{n}$ is not $\uoLt_E$-convergent to zero in $E$.
\end{example}

\begin{example}\label{exam:counterexample_5}
	We note that \cref{exam:counterexample_4} also gives an example of a \emph{uniformly bounded sequence $\seq{T_n}{n}$ of order continuous operators on a separable reflexive Banach lattice $E$ with a weak order unit that admits a Hausdorff \uoLtop, such that $T_nx\conv{\uoLt_E} 0$ in $E$ for all $x\in E$ but $T_n\xslashedrightarrowc{\uoLt_{\regops(E)}} 0$ in $\regops(E)$.}	
	\end{example}

\section{Orthomorphisms}\label{sec:orthomorphisms}

\noindent In this section, we review some material concerning orthomorphism and establish a few auxiliary result for use in the present paper and in future ones.

Let $E$ be a vector lattice. We recall from \cite[Definition~2.41]{aliprantis_burkinshaw_POSITIVE_OPERATORS_SPRINGER_REPRINT:2006} that an operator on $E$ is called an \emph{orthomorphism} when it is a band preserving order bounded operator. An orthomorphism is evidently disjointness preserving, it is order continuous (see \cite[Theorem~2.44]{aliprantis_burkinshaw_POSITIVE_OPERATORS_SPRINGER_REPRINT:2006}), and its kernel is a band (see \cite[Theorem~2.48]{aliprantis_burkinshaw_POSITIVE_OPERATORS_SPRINGER_REPRINT:2006}). We denote by $\Orth(E)$ the collection of all orthomorphism on $E$.  Even when $E$ is not Dedekind complete, the supremum and infimum of two orthomorphisms $S$ and $T$ in $E$ always exists in $\obops(E)$. In fact, we have
\begin{equation}
\begin{split}\label{eq:sup_of_two_orthormorphisms}
\left[S\vee T\right](x)&=S(x)\vee T(x)\\
\left[S\wedge T\right](x)&=S(x)\wedge T(x)
\end{split}
\end{equation}
for $x\in\posE$ and
\begin{equation}\label{eq:modulus_of_orthormorphism}
\abs{Tx}=\abs{T}\abs{x}=\abs{T(\abs{x})}
\end{equation}
for $x\in E$; see \cite[Theorems~2.43 and~2.40]{aliprantis_burkinshaw_POSITIVE_OPERATORS_SPRINGER_REPRINT:2006}. Consequently, $\Orth(E)$ is a unital vector lattice algebra for every vector lattice $E$. Even more is true: according to \cite[Theorem~2.59]{aliprantis_burkinshaw_POSITIVE_OPERATORS_SPRINGER_REPRINT:2006}, $\Orth(E)$ is an (obviously Archimedean) $f$-algebra for every vector lattice $E$, so it is commutative by  \cite[Theorem~2.56]{aliprantis_burkinshaw_POSITIVE_OPERATORS_SPRINGER_REPRINT:2006}.  Furthermore, for every vector lattice $E$, when $T\in\Orth(E)$ and $T:E\to E$ is injective and surjective, then the linear map $T^{-1}:E\to E$ is again an orthomorphism. We refer to \cite[Theorem~3.1.10]{meyer-nieberg_BANACH_LATTICES:1991} for a proof of this result of Huijsmans' and de Pagter's.

It follows easily from \cref{eq:sup_of_two_orthormorphisms} that, for every vector lattice $E$, the identity operator is a weak order unit of $\Orth(E)$. When $E$ is Dedekind complete, $\Orth(E)$ is the band in $\regops(E)$ that is generated by the identity operator on $E$;  see \cite[Theorem 2.45]{aliprantis_burkinshaw_POSITIVE_OPERATORS_SPRINGER_REPRINT:2006}.

Let $E$ be a vector lattice, let $T\in\obops(E)$, and let $\lambda\geq 0$. Using \cite[Theorem~2.40]{aliprantis_burkinshaw_POSITIVE_OPERATORS_SPRINGER_REPRINT:2006}, it is not difficult to see that the following are equivalent:
\begin{enumerate}
	\item $-\lambda\Id \leq T\leq \lambda\Id$;
	\item $\abs{T}$ exists in $\obops(E)$, and $\abs{T}\leq\lambda I$;
	\item $\abs{Tx}\leq\lambda\abs{x}$ for all $x\in E$.
	\end{enumerate}
The set of all such $T$ is a unital subalgebra $\mathcal Z(E)$ of $\Orth(E)$ consisting of ideal preserving order bounded operators on $E$. It is called the \emph{ideal centre of $E$}.

Let $E$ be a vector lattice, and define the \emph{stabiliser of $E$}, denoted by $\mathcal S(E)$, as the set of linear operators on $E$ that are ideal preserving. It is not required that these operators be order bounded, but this is nevertheless always the case. In fact, $\mathcal S(E)$ is a unital subalgebra of $\Orth(E)$ for every vector lattice $E$ (see \cite[Proposition~2.6]{wickstead:1977a}), so that we have the chain
\[
\mathcal Z(E)\subseteq\mathcal S(E)\subseteq\Orth(E)
\]
of unital algebras for every vector lattice $E$. For every Banach lattice $E$, we have
\[
\mathcal Z(E)=\mathcal S(E)=\Orth(E);
\]
see \cite[Corollary~4.2]{wickstead:1977a}, so that the identity operator on $E$ is then even an order unit of $\Orth(E)$.

For every Banach lattice $E$, $\Orth(E)$ is a unital Banach subalgebra of the bounded linear operators on $E$ in the operator norm. This follows easily from the facts that bands are closed and that a band preserving operator on a Banach lattice is automatically order bounded; see \cite[Theorem~4.76]{aliprantis_burkinshaw_POSITIVE_OPERATORS_SPRINGER_REPRINT:2006}.

Let $E$ be a Banach lattice. Since the identity operator is an order unit of $\Orth(E)$, we can introduce the order unit norm $\norm{\,\cdot\,}_\Id$ with respect to $\Id$ on $\Orth(E)$ by setting
\[
\norm{T}_\Id\coloneqq \inf\{\,\lambda\geq 0: \abs{T}\leq \lambda\Id\,\}
\]
for $T\in\Orth(E)$. Then $\norm{T}=\norm{T}_\Id$ for all $T\in\Orth(E)$; see \cite[Proposition~4.1]{wickstead:1977a}. Since we already know that $\Orth(E)$ is complete in the operator norm, it follows that $\Orth(E)$, when supplied with $\norm{\,\cdot\,}=\norm{\,\cdot\,}_\Id$, is a unital Banach lattice algebra that is also an AM-space. When $E$ is a Dedekind complete Banach lattice, then evidently $\norm{T}=\norm{T}_I=\norm{\abs{T}}_I=\norm{\,\abs{T}\,}=\norm{T}_\reg$ for $T\in\Orth(E)$. Hence $\Orth(E)$ is then  also a unital Banach lattice subalgebra of the Banach lattice algebra of all order bounded operators on $E$ in the regular norm.

Let $E$ be Banach lattice. It is clear from the above that  $(\Orth(E),\norm{\,\cdot\,})=(\Orth(E),\norm{\,\cdot\,}_I)$ is a unital Banach $f$-algebra in which its identity element is also a (positive) order unit. The following result is, therefore, applicable with $\alg=\Orth(E)$ and $e=\Id$. It shows, in particular, that $\Orth(E)$ is isometrically Banach lattice algebra isomorphic to a $C(K)$-space. Both its statement and its proof improve on the ones in \cite[Proposition~2.6]{de_jeu_wortel:2014}, \cite[Proposition~1.4]{schaefer_wolff_arendt:1978}, and \cite{hackenbroch:1977}.

\begin{theorem}\label{res:f-algebra_is_C(K)-space}
Let $\alg$ be a unital $f$-algebra such that its identity element $e$ is also a \uppars{positive} order unit, and such that it is complete in the submultiplicative order unit norm $\norm{\,\cdot\,}_e$ on $\alg$. Let $\algtwo$ be a \uppars{not necessarily unital} associative subalgebra of $\alg$. Then $\overline{\algtwo}^{\norm{\,\cdot\,}_e}$ is a Banach $f$-subalgebra of $\alg$. When $e\in\overline{\algtwo}^{\norm{\,\cdot\,}_e}$, then there exist a compact Hausdorff space $K$, uniquely determined up to homeomorphism, and an isometric surjective Banach lattice algebra isomorphism $\psi: \overline{\algtwo}^{\norm{\,\cdot\,}_e}\to C(K)$.
\end{theorem}

\begin{proof}
Since $(\alg, \norm{\,\cdot\,}_I)$ is an AM-space with order unit $e$, there exist a compact Hausdorff space $K^\prime$ and an isometric surjective lattice homomorphism $\psi^\prime:\alg\to C(K^\prime)$ such that $\psi^\prime(e)=1$; see \cite[Theorem~2.1.3]{meyer-nieberg_BANACH_LATTICES:1991} for this result of Kakutani's, for example. Via this isomorphism, the $f$-algebra multiplication on $C(K^\prime)$ provides the vector lattice $\alg$ with a multiplication that makes $\alg$ into an $f$-algebra with $e$ as its positive multiplicative identity element. Such a multiplication is, however, unique; see \cite[Theorem~2.58]{aliprantis_burkinshaw_POSITIVE_OPERATORS_SPRINGER_REPRINT:2006}. Hence $\psi^\prime$ also preserves multiplication, and we conclude that $\psi^\prime:\alg\to C(K^\prime)$ is an isometric surjective Banach lattice algebra isomorphism.

We now turn to $\algtwo$. It is clear that $\overline{\algtwo}^{\norm{\,\cdot\,}_e}$ is Banach subalgebra of $\alg$. After moving to the $C(K^\prime)$-model for $\alg$ that we have obtained, \cite[Lemma~4.48]{folland_REAL_ANALYSIS_SECOND_EDITION:1999} shows that $\overline{\algtwo}^{\norm{\,\cdot\,}_e}$ is also a vector sublattice of $\alg$. Hence $\overline{\algtwo}^{\norm{\,\cdot\,}_e}$ is a Banach $f$-subalgebra of $\alg$. When $e\in\overline{\algtwo}^{\norm{\,\cdot\,}_e}$, we can then apply the first part of the proof to $\overline{\algtwo}^{\norm{\,\cdot\,}_e}$, and obtain a compact Hausdorff space $K$ and an isometric surjective Banach lattice algebra isomorphism $\psi: \overline{\algtwo}^{\norm{\,\cdot\,}_e}\to C(K)$.  The Banach--Stone theorem (see \cite[Theorem~VI.2.1]{conway_A_COURSE_IN_FUNCTIONAL_ANALYSIS_SECOND_EDITION:1990}, for example) implies that $K$ is uniquely determined up to homeomorphism.
\end{proof}

We now proceed to show that $E$ and $\Orth(E)$ have isomorphic universal completions. We start with a preparatory lemma.

\begin{proposition}\label{res:two_ideals_isomorphic}
Let $E$ be a Dedekind complete vector lattice, and let $x\in E$. Let $I_x$ be the principal ideal of $E$ that is generated by $x$, let $B_x$ be the principal band in $E$ that is generated by $x$, let $P_x: E\to B_x$ be the corresponding order projection, and let $\opideal_{P_x}$ be the principal ideal of $\regops(E)$ that is generated by $P_x$. For $T\in \opideal_{P_x}$, set $\psi_x(T)\coloneqq T\abs{x}$. Then $\psi_x(T)\in I_x$, and:
\begin{enumerate}
	\item the map  $\psi_x: \opideal_{P_x}\to I_x$ is a surjective vector lattice isomorphism such that $\psi_x(P_x)=\abs{x}$;
	\item $\opideal_{P_x}=P_x\centre(E)$.
\end{enumerate}	
\end{proposition}

\begin{proof}
		Take $T\in\opideal_{P_x}$. There exists a $\lambda\geq 0$ such that $\abs{T}\leq \lambda P_x$, and this implies that $\abs{Ty}\leq \lambda P_x\abs{y}$ for all $y\in E$. This shows that $T\abs{x}\in I_x$, so that $\psi_x$ maps $\opideal_{P_x}$ into $I_x$; it also shows that $T(B_x^\mathrm{d})=\{0\}$. Suppose that $T\abs{x}=0$. Since the kernel of $T$ is a band in $E$, this implies that $T$ vanishes on $B_x$. We already know that it vanishes on $B_x^{\mathrm d}$. Hence $T=0$, and we conclude that $\psi_x$ is injective. We show that $\psi_x$ is surjective. Let $y\in I_x$. Take a $\lambda>0$ such that $0\leq \abs{y/\lambda}\leq\abs{x}$. An inspection of the proof of \cite[Theorem~2.49]{aliprantis_burkinshaw_POSITIVE_OPERATORS_SPRINGER_REPRINT:2006} shows that there exists a $T\in\centre(E)$ with $T\abs{x}=y/\lambda$. Since $\lambda T P_x \in\opideal_{P_x}$ and $(\lambda T P_x) \abs{x}=y$, we see that $\psi_x$ is surjective. Finally, it is clear from \cref{eq:sup_of_two_orthormorphisms} that $\psi_x$ is a vector lattice homomorphism. This completes the proof of part~(1).

	We turn to part~(2). It is clear that $\opideal_{P_x}\supseteq P_x\centre(E)$. Take $T\in\opideal_{P_x}\subseteq\centre(E)$. Then also $P_xT\in\opideal_{P_x}$. Since $\psi_x(T)=\psi_x(P_xT)$, the injectivity of $\psi_x$ on $\opideal_{P_x}$ implies that $T=P_xT\in P_x\centre(E)$.
	\end{proof}

The first part of \cref{res:two_ideals_isomorphic} is used in the proof of our next result.

\begin{proposition}\label{res:isomorphic_order_dense_ideals}
	Let $E$ be a Dedekind complete  vector lattice. Then there exist an order dense ideal $I$ of $E$ and an order dense ideal $\opideal$ of $\Orth(E)$ such that $I$ and $\opideal$ are isomorphic vector lattices.
\end{proposition}

\begin{proof}
	Choose a maximal disjoint system $\{\,x_\alpha : \alpha\in\is\,\}$ in $E$. For each $\alpha\in\is$, let $I_{x_\alpha}$, $B_{x_\alpha}$, $P_{x_\alpha}: E\to B_{x_\alpha}$, $\opideal_{P_{x_\alpha}}$, and the vector lattice isomorphism $\psi_{x_\alpha}: \opideal_{P_{x_\alpha}}\to I_{x_\alpha}$ be as in \cref{res:two_ideals_isomorphic}.
	
	Since the $x_\alpha$'s are mutually disjoint, it is clear that the ideal $\sum_{\alpha\in\is} I_{x_\alpha}$ of $E$ is, in fact, an internal direct sum $\bigoplus_{\alpha\in\is} I_{x_\alpha}$. Since the disjoint system is maximal, $\bigoplus_{\alpha\in\is} I_{x_\alpha}$ is an order dense ideal of $E$.
	
	It follows easily from \cref{eq:sup_of_two_orthormorphisms} that the $P_{x_\alpha}$ are also mutually disjoint. They even form a maximal disjoint system in $\Orth(E)$. To see this, suppose that $T\in\Orth(E)$ is such that $\abs{T}\wedge P_{x_\alpha}=0$ for all $\alpha\in\is$. Then $(\abs{T}x_\alpha)\wedge x_\alpha=(\abs{T}\wedge P_{x_\alpha})x_\alpha=0$ for all $\alpha\in\is$. Since $\abs{T}$ is band preserving, this implies that $\abs{T}x_\alpha=0$ for all $\alpha\in\is$. The fact that the kernel of $\abs{T}$ is a band in $E$ then yields that $\abs{T}=0$. Just as for $E$, we now conclude that the ideal $\sum_{\alpha\in\is} \opideal_{P_{x_\alpha}}$ of $\Orth(E)$ is an internal direct sum $\bigoplus_{\alpha\in\is} \opideal_{P_{x_\alpha}}$ that is order dense in $\Orth(E)$.
	
	Since  $\bigoplus_{\alpha\in\is}\psi_{x_{\alpha}}:\bigoplus_{\alpha\in\is} \opideal_{P_{x_\alpha}}\to\bigoplus_{\alpha\in\is} I_{x_\alpha}$ is a vector lattice isomorphism by \cref{res:two_ideals_isomorphic}, the proof is complete.
	\end{proof}

It is generally true that a vector lattice and an order dense vector sublattice of it have isomorphic universal completions; see \cite[Theorems~7.21 and~7.23]{aliprantis_burkinshaw_LOCALLY_SOLID_RIESZ_SPACES_WITH_APPLICATIONS_TO_ECONOMICS_SECOND_EDITION:2003}. \cref{res:isomorphic_order_dense_ideals} therefore implies the following.

\begin{corollary}\label{res:isomorphic_universal_completions}	
	Let $E$ be a Dedekind complete vector lattice. Then the universal completions of $E$ and of $\Orth(E)$ are isomorphic vector lattices.
\end{corollary}

The previous result enables us to relate the countable sup property of $E$ to that of $\Orth(E)$. We recall that  vector lattice $E$ \emph{has the countable sup property} when, for every non-empty subset $S$ of $E$ that has a supremum in $E$, there exists an at most countable subset of $S$ that has the same supremum in $E$ as $S$.
In parts of the literature, such as in \cite{luxemburg_zaanen_RIESZ_SPACES_VOLUME_I:1971} and \cite{zaanen_INTRODUCTION_TO_OPERATOR_THEORY_IN_RIESZ_SPACES:1997}, $E$ is then said to be order separable. We also recall that a subset of a vector lattice is said to be an \emph{order basis} when the band that it generates is the whole vector lattice.

\begin{proposition} Let $E$ be a Dedekind complete vector lattice. The following are equivalent:
	\begin{enumerate}
		\item $\Orth(E)$ has the countable sup property;
		\item $E$ has the countable sup property and an at most countably infinite order basis.
	\end{enumerate}	
\end{proposition}

\begin{proof}
It is proved in \cite[Theorem~6.2]{kandic_taylor:2018} that, for an arbitrary vector lattice $F$, $F^\mathrm{u}$ has the countable sup property if and only if $F$ has the countable sup property as well as an at most countably infinite order basis. Since $\Orth(E)$ has a weak order unit $\Id$, we see that $\Orth(E)^\mathrm{u}$ has the countable sup property if and only if $\Orth(E)$ has the countable sup property. On the other hand, since $\Orth(E)^\mathrm{u}$ and $E^\mathrm{u}$ are isomorphic by \cref{res:isomorphic_universal_completions}, an application of this same result to $E$ shows that $\Orth(E)^\mathrm{u}$ has the countable sup property if and only if $E$ has the countable sup property and an at most countably infinite order basis.
\end{proof}

	We shall now establish a uniform order boundedness principle for orthomorphisms. It will be needed in the proof of \cref{res:oRKF}, below. 
	
	\begin{proposition}\label{res:UBP}
		Let $E$ be a Dedekind complete vector lattice, and let $\{\,T_\alpha : \alpha\in \is\,\}$ be a non-empty subset of $\Orth(E)$. The following are equivalent:
		\begin{enumerate}
			\item $\{\,T_\alpha : \alpha\in \is\,\}$ is an order bounded subset of $\regops(E)$;
			\item for each $x\in E$, $\{\, T_\alpha x : \alpha\in \is\,\}$ is an order bounded subset of $E$.
		\end{enumerate}
		
	\end{proposition}
	
	Before proceeding with the proof, we remark that, since $\Orth(E)$ is a projection band in $\regops(E)$, the order boundedness of the net could equivalently have been required in $\Orth(E)$.

	\begin{proof} It is trivial that part~(1) implies part~(2). We now show the converse. Take an $x\in\pos{E}$. The hypothesis in part~(2), together with \cref{eq:modulus_of_orthormorphism}, shows that  $\{\, \abs{T_\alpha} x : \alpha\in \is\,\}$ is an order bounded subset of $E$. Hence the same is true for $\{\, \abs{T_\alpha} x : \alpha\in \is\,\}^{\vee}$ which, in view of \cref{eq:sup_of_two_orthormorphisms}, equals $\{\,Sx : S \in \{\,\abs{T_\alpha}: \alpha\in \is\,\}^\vee\,\}$. Using \cite[Theorem~1.19]{aliprantis_burkinshaw_POSITIVE_OPERATORS_SPRINGER_REPRINT:2006}, we conclude that $\{\,S : S \in \{\,\abs{T_\alpha}: \alpha\in \is\,\}^\vee\,\}$ is bounded above in $\obops(E)$. Then the same is true for $\{\,\abs{T_\alpha} : \alpha\in \is\,\}$, as desired.
\end{proof}

\cref{res:UBP} fails for nets of general order bounded operators. It can, in fact, already fail for a sequence of order continuous operators on a Banach lattice, as is shown by the following example. 

	\begin{example}\label{exam:UBP_fail_in_obops}
		Let $E\coloneqq\ell_{\infty}(\NN)$, and let $(e_i)_{i=1}^{\infty}$ be the standard unit vectors in $E$. Let $S\in\ocops(E)$ be the right shift, and set $T_n\coloneqq S^n$ for $n\geq 1$. It is easy to see that $\seq{T_n x}{n}$ is order bounded in $E$ for all $x\in E$. %Indeed, take $x=\bigvee_{i=1}^\infty \lambda_n e_n$ with $M=\sup_n \abs{\lambda_n}$, then $\{T_n x:n\in \NN\}\subseteq [-Me, Me]$, where $e=\bigvee_{n}e_n$.
		We shall show, however, that $\seq{T_n}{n}$ is not order bounded in $\obops(E)$. To see this, we first note that $T_m\perp T_n$ for $m,n\geq 1$ with $m\neq n$. Indeed, for all $i\geq 1$, we have $0\leq (T_m\wedge T_m)e_i\leq T_m(e_i)\wedge T_n(e_i)=e_{m+i}\wedge e_{n+i}=0$. Hence $(T_m\wedge T_n)x=0$ for all $x\in I$, where $I$ is the ideal of $E$ that is spanned by $\{\,e_i:i\geq 1\,\}$. Since $I$ is order dense in $E$ and $T_n\wedge T_m\in \ocops(E)$, it follows that $T_n\wedge T_m=0$ for all $m,n\geq 1$ with $m\neq n$. Suppose that $T$ is an upper bounded of $\seq{T_n}{n}$ in $\obops(E)$. Set $e\coloneqq\bigvee_{i=1}^\infty e_i$. Using the disjointness of the $T_n$, we have 
		\[
		Te\geq \left(\bigvee_{i=1}^{n}T_i\right)e=\left(\sum_{i=1}^{n}T_i\right)e\geq ne_{n+1}
		\]
		for all $n\geq 1$, which is impossible. So $\seq{T_n}{n}$ is not order bounded in $\obops(E)$.
\end{example}

As a side result, we note the following consequence of \cref{res:UBP}. It is an ordered analogue of the familiar result for a sequence of bounded operators on a Banach space.

\begin{corollary}\label{res:coro_Tn}
	Let $E$ be a Dedekind complete vector lattice, and let $\seq{T_n}{n}$ be a sequence in $\Orth(E)$. Suppose that the sequence $\seq{T_nx}{n}$ is order convergent in $E$ for all $x\in E$. Then $\{\,T_n : n\geq 1\,\}$ is an order bounded subset of $\regops(E)$. For $x\in E$, define $T:E\to E$ by setting
	\[
	Tx\coloneqq \mathrm{o}\text{ --}\lim_{n\to\infty} T_n x.
	\]
	Then $T\in\Orth(E)$.
\end{corollary}

\begin{proof}
	Using \cref{res:UBP}, it is clear that $T$ is a linear and order bounded operator on $E$. Since each of the $T_n$ is a band preserving operator, the same is true for $T$. Hence $T$ is an orthomorphism on $E$.%Since order convergent sequences are order bounded, \cref{res:UBP} shows that there exist an $S\in\Orth(E)$ such that $\abs{T_n}\leq\abs{S}$ for $n\geq 1$. As $\Orth(E)=\centre(E)$, there exists a $\lambda\geq 0$ such that $\abs{T_n}\leq \lambda I$ for $n\geq 1$. Using \cref{eq:modulus_of_orthormorphism}, one then easily sees that $\abs{Tx}\leq\lambda\abs{x}$ for $x\in E$. Hence $T\in\centre(E)=\Orth(E)$.
\end{proof}

We conclude by giving some estimates for orthomorphisms that will be used in the sequel. As a preparation, we need the following extension of \cite[Exercise~1.3.7]{aliprantis_burkinshaw_POSITIVE_OPERATORS_SPRINGER_REPRINT:2006}.

\begin{lemma}\label{res:projection_inequality}
	Let $E$ be a vector lattice with the principal projection property. Take $x,y\in E$. For  $\lambda\in\RR$, let $P_\lambda$ denote the order projection in $E$ onto the band generated by $(x-\lambda y)^+$. Then $\lambda P_\lambda y\leq P_\lambda x$. When $x,y\in\posE$ and $\lambda\geq 0$, then $x\leq\lambda y+P_\lambda x$.
	\end{lemma}

\begin{proof}
	The first inequality follows from the fact that
	\[
	0\leq P_\lambda (x-\lambda y)^+=P_\lambda (x-\lambda y)=P_\lambda x-\lambda P_\lambda y.
	\]
 For the second inequality, we note that   $x-\lambda y\leq (x-\lambda y)^+=P_\lambda(x-\lambda y)^+$ for all $x$, $y$, and $\lambda$. When $x,y\in\posE$ and $\lambda\geq 0$, then $(x-\lambda y)^+\leq x^+=x$, so that
 \[
 x\leq \lambda y+P_\lambda(x-\lambda y)^+\leq \lambda y + P_\lambda x.
 \]
\end{proof}

\begin{proposition}\label{res:estimates_for_order_projections_in_orthomorphisms}
	\!Let $E$ be a Dedekind complete vector lattice, and let\! $T\!\!\in\!\pos{\Orth(E)}$\!\!. For $\lambda>0$, let $\mathcalalt P_\lambda$ be the order projection in $\Orth(E)$ onto the band generated by $(T-\lambda\Id)^+$ in $\Orth(E)$. There exists a unique order projection $P_\lambda$ in $E$ such that $\mathcalalt P_\lambda(S)=P_\lambda S$ for all $S\in\Orth(E)$. Furthermore:
	\begin{enumerate}
		\item $\lambda P_\lambda\leq P_\lambda T\leq T$;
		\item $T\leq \lambda\Id + P_\lambda T$;
		\item $(P_\lambda T x)\wedge y\leq \frac{1}{\lambda} Ty$ for all $x,y\in\posE$.
	\end{enumerate}
\end{proposition}

\begin{proof}
Since $0\leq \mathcalalt P_\lambda\leq \Id_{\Orth(E)}$, it follows from \cite[Theorem~2.62]{aliprantis_burkinshaw_POSITIVE_OPERATORS_SPRINGER_REPRINT:2006} that there exists a unique $P_\lambda\in\Orth(E)$ with $0\leq P_\lambda\leq \Id$ such that $\mathcalalt P_\lambda (S)=P_\lambda S$ for all $S\in\Orth(E)$. The fact that $\mathcalalt P_\lambda$ is idempotent implies that $P_\lambda$ is also idempotent. Hence $P_\lambda$ is an order projection.

The inequalities in the parts~(1) and~(2) are then a consequence of those in \cref{res:projection_inequality}. For part~(3), we note that $(P_\lambda T x)\wedge y$ is in the image of the projection $P_\lambda$. Since order projections are vector lattice homomorphisms, we have, using part~(1) in the final step, that
\[
 (P_\lambda T x)\wedge y=P_\lambda((P_\lambda T x)\wedge y)=(P_\lambda^2 Tx)\wedge P_\lambda y\leq P_\lambda y\leq \frac{1}{\lambda} Ty.
\]
\end{proof}

We shall have use for the following corollary, which has some appeal of its own.

\begin{corollary}\label{res:interesting_ineq}Let $E$ be a Dedekind complete vector lattice, and let $T\in\pos{\Orth(E)}$. Then
\[(Tx) \wedge y\leq \lambda(x\wedge y)+\frac{1}{\lambda}Ty\]
for all $x,y\in E^+$ and $\lambda>0$.
\end{corollary}
\begin{proof}For $\lambda>0$, we let $\mathcalalt P_\lambda$ be the order projection in $\Orth(E)$ onto the band generated by $(T-\lambda\Id)^+$ in $\Orth(E)$. According to \cref{res:estimates_for_order_projections_in_orthomorphisms}, there exists a unique order projection $P_\lambda$ in $E$ such that $\mathcalalt P_\lambda(S)=P_\lambda S$ for all $S\in\Orth(E)$. By applying part~(2) of \cref{res:estimates_for_order_projections_in_orthomorphisms} in the first step and its part~(3) in the third, we have, for $x,y\in\pos{E}$,
\begin{align*}
    (Tx) \wedge y &\leq (\lambda x+P_\lambda T x)\wedge y\\
& \leq \lambda(x\wedge y)+P_\lambda T x \wedge y\\
& \leq \lambda(x\wedge y)+ \frac{1}{\lambda}Ty.
\end{align*}
\end{proof}

\section{Continuity properties of orthomorphisms}\label{sec:continuity_of_orthomorphisms}

\noindent Orthomorphisms preserve order convergence of nets. In this short section, we show that they also preserve unbounded order convergence and\textemdash when applicable\textemdash convergence in the Hausdorff uo-Lebesgue topology.

Before doing so, let us note that this is in contrast to the case of general order bounded operators. Surely, there exist order bounded operators that are not order continuous. For the remaining two  convergence structures, we consider $\ell_1$ with its standard basis $\seq{e_n}{n}$. It follows from \cite[Corollary~3.6]{gao_troitsky_xanthos:2017} that $e_n\uoconv 0$. There are several ways to see that $\ell_1$ admits a \necun\ Hausdorff \uoLtop\ $\uoLt_{\ell_1}$. This follows from the fact that its norm is order continuous (see \cite[p.~993]{taylor:2019}), from the fact that it is atomic (see \cite[Lemma~7.4]{taylor:2019}), and from a result in the context of measure spaces (see \cite[Theorem~6.3]{deng_de_jeu_UNPUBLISHED:2020a}). The latter two results also show that $\uoLt_{\ell_1}$ is the topology of coordinatewise convergence. In particular, $e_n\conv{\uoLt_{\ell_1}}0$ which is, of course, also a consequence of the fact that $e_n\uoconv 0$.  Define $T:\ell_1\to\ell_1$ by setting $Tx\coloneqq \left(\sum_{n=1}^\infty x_n\right)e_1$ for $x=\sum_{n=1}^\infty x_ne_n\in\ell_1$. Since $Te_n=e_1$ for all $n\geq 1$, the order continuous positive operator $T$ on $\ell_1$ preserves neither uo-convergence nor $\uoLt_{\ell_1}$-convergence of sequences in $\ell_1$.

\begin{proposition}\label{res:uocnt}
	Let $E$ be a Dedekind complete vector lattice, and let $T \in \Orth(E)$. Suppose that $\net$ is a net in $E$ such that $x_\alpha\uoconv 0$ in $E$. Then $Tx_\alpha\uoconv 0$ in $E$.
\end{proposition}

\begin{proof} Using \cref{eq:modulus_of_orthormorphism}, one easily sees that we may suppose that $T$ and the $x_\alpha$'s are positive. Let $B_{T(E)}$ denote the band in $E$ that is generated by $T(E)$. Take a $y\in \pos{T(E)}$. Since a positive orthomorphism is a lattice homomorphism, there exists an $x\in\pos{E}$ such that $y=Tx$. Using the fact that $x_\alpha\uoconv 0$ in $E$, the order continuity of $T$ then implies that	
\[
Tx_\alpha\wedge y=Tx_\alpha\wedge Tx=T(x_\alpha\wedge x)\oconv 0
\]
in $E$. Then \cite[Corollary~2.12]{gao_troitsky_xanthos:2017} shows that also $Tx_\alpha\wedge y\oconv 0$ in the regular vector sublattice $B_{T(E)}$ of $E$. Since $B_{T(E)}$ also equals the band in $B_{T(E)}$ that is generated by $T(E)$, an appeal to \cite[Lemma~2.2]{li_chen:2018} yields that $Tx_\alpha\uoconv 0$ in $B_{T(E)}$. Hence $Tx_\alpha\wedge \abs{y}\oconv 0$ in $B_{T(E)}$ for all $y\in B_{T(E)}$, and then also $Tx_\alpha\wedge \abs{y}\oconv 0$ in $E$ for all $y\in B_{T(E)}$. Since $E=B_{T(E)}\oplus \big(B_{T(E)}\big)^\dc$, it is now clear that 
$Tx_\alpha\wedge \abs{y}\oconv 0$ in $E$ for all $y\in E$.
\end{proof} 

For the case of a Hausdorff \uoLtop, we need the following preparatory result that has some independent interest. \cref{res:oLt_preparation} is of the same flavour.

\begin{proposition}\label{res:tauctn}
Let $E$ be a Dedekind complete vector lattice that admits a \uppars{not necessarily Hausdorff} locally solid linear topology $\tau_E$, and let $T \in \Orth(E)$. Suppose that $\net$ is a net in $E$ such that $x_\alpha\conv{\tau_E} 0$ in $E$. Then $Tx_\alpha\conv{\unb_E\tau_E} 0$ in $E$.	
	\end{proposition}

\begin{proof} As in the proof of \cref{res:uocnt}, we may suppose that $T$ and the $x_\alpha$ are positive. For $n\geq 1$, we let $\mathcalalt P_n$ be the order projection in $\Orth(E)$ onto the band generated by $(T-n\Id)^+$ in $\Orth(E)$ again, so that again there exists a unique order projection $P_n$ in $E$ such that $\mathcalalt P_n(S)=P_nS$ for all $S\in\Orth(E)$.  Fix $e\in E^+$. Take a solid $\tau_E$-neighbourhood $U$ of $0$ in $E$, and choose a $\tau_E$-neighbourhood $V$ of $0$ such that $V+V\subseteq U$. Take an $n_0\geq 1$ such that $T e/n_0\in V$. As $x_\alpha\conv{\tau_E} 0$, there exists an $\alpha_0\in\is$ such that $n_0x_\alpha\in V$ for all $\alpha\geq\alpha_0$. By applying \cref{res:interesting_ineq} in the first step, we have, for all $\alpha\geq\alpha_0$, 
\begin{equation}\label{eq:estimates_for_tau_continuity_of_orthomorphisms}
\begin{split}
(Tx_\alpha) \wedge e & \leq n_0(x_\alpha\wedge e)+ \frac{1}{n_0}Te\\
& \leq n_0x_\alpha + \frac{1}{n_0}T e\\
&\in V+V\subseteq U
\end{split}
\end{equation}
The solidness of $V$ then implies that $(Tx_\alpha) \wedge e\in U$ for all $\alpha\geq \alpha_0$. Since $U$ and $e$ were arbitrary, we conclude that $T_\alpha x \conv{\unb_E\tau_E} 0$.
\end{proof}

Since the unbounded topology $\unb_E\uoLt_E$ that is generated by a Hausdorff \uoLtop\ $\uoLt_E$ equals $\uoLt_E$ again, the following is now clear.

\begin{corollary}\label{res:uoLtctn}
	Let $E$ be a Dedekind complete vector lattice that admits a \necun\ Hausdorff \uoLtop\ $\uoLt_E$, and let $T \in \Orth(E)$. Suppose that $\net$ is a net in $E$ such that $x_\alpha\conv{\uoLt_E} 0$ in $E$. Then $Tx_\alpha\conv{\uoLt_E} 0$ in $E$.	
\end{corollary}

\section{Topologies on $\Orth(E)$}\label{sec:topologies_on_the_orthomorphisms_of_a_vector_lattice}

\noindent Let $E$ be a Dedekind complete vector lattice, and suppose that $\tau_E$ is a (not necessarily Hausdorff) locally solid additive topology on $E$. Take a non-empty subset $S$ of $E$. According to \cref{res:topology_on_order_bounded_operators}, there exists a unique additive topology $\optop{S}\tau_E$ on $\regops(E)$ such that, for a net $\opnet$ in $\regops(E)$, $T_\alpha\conv{\optop{S}\tau_E} 0$ if and only if $\abs{T_\alpha}\abs{s}\conv{\tau_E} 0$ for all $s\in S$. When $\opnet\subseteq\Orth(E)$, \cref{eq:modulus_of_orthormorphism} and the local solidness of $\tau_E$ imply that this convergence criterion is also equivalent to the one that $T_\alpha s\conv{\tau_E}0$ for all $s\in S$. Hence \emph{on subsets of $\Orth(E)$, an absolute strong operator topology that is generated by a locally solid additive topology on $E$ coincides with the corresponding strong operator topology}. In order to remind ourselves of the connection with the topology on the enveloping vector lattice $\regops(E)$ of $\Orth(E)$, we shall keep writing $\optop{S}\tau_F$ when considering the restriction of this topology to subsets of $\Orth(E)$, rather than switch to, e.g., $\mathrm{SOT}_S\tau_F$.

The above observation can be used in several results in \cref{sec:topologies_on_vector_lattices_of_order_bounded_operators}. For the ease of reference, we include the following consequence of \cref{res:topology_on_order_bounded_operators_practical}.

\begin{corollary}\label{res:topology_on_orthomorphism_practical}
	Let $E$ be a Dedekind complete vector lattice, and let $\tau_E$ be a \uppars{not necessarily Hausdorff} locally solid linear topology on $E$. Take a vector sublattice  $\oplat$ of $\Orth(E)$ and a non-empty subset $S$ of $E$.
	
	There exists a unique locally solid linear topology $\optop{S}\tau_E$ on $\oplat$ such that, for a net $\opnet$ in $\oplat$, $T_\alpha\conv{\optop{S}\tau_E}0$ if and only if $T_\alpha s\conv{\tau_E}0$ for all $s\in S$.
	
	Let $I_S$ be the ideal of $E$ that is generated by $S$.
	For a net $\opnet$ in $\oplat$, $T_\alpha\!\conv{\optop{S}\tau_E}0$ if and only if $T_\alpha x\conv{\tau_E}0$ for all $x\in I_S$.
	
	When $\tau_E$ is a Hausdorff topology on $E$, then $\optop{S}\tau_E$ is a Hausdorff topology on $\oplat$ if and only if $I_S$ separates the points of $\oplat$.\label{part:topology_on_orthomorphisms_practical_2}

\end{corollary}

According to the next result, there is an intimate relation between the existence of Hausdorff  \oLtops\ and \uoLtops\ on $E$ and on $\Orth(E)$.

\begin{proposition}\label{res:oLtop_and_uoLtop_on_E_and_Orth(E)}
	Let $E$ be a Dedekind complete vector lattice. The following are equivalent:
	\begin{enumerate}
		\item $E$ admits a Hausdorff \oLtop;
		\item $\Orth(E)$ admits a Hausdorff \oLtop;
		\item $E$ admits a \necun\ Hausdorff \uoLtop;
		\item $\Orth(E)$ admits a \necun\ Hausdorff \uoLtop.
	\end{enumerate}
\end{proposition}

\begin{proof}
	As $E$ and $\Orth(E)$ are Dedekind complete, they are not just order dense vector sublattices of their universal completions but even order dense ideals; see \cite[p.126--127]{aliprantis_burkinshaw_POSITIVE_OPERATORS_SPRINGER_REPRINT:2006}. Since these universal completions are isomorphic vector lattices by \cref{res:isomorphic_universal_completions}, the proposition follows from a double application of \cite[Theorem~4.9.(3)]{deng_de_jeu_UNPUBLISHED:2020a}.
	\end{proof}

For a Dedekind complete vector lattice $E$, $\Orth(E)$, being a band in $\regops(E)$, is a regular vector sublattice of $\regops(E)$. A regular vector sublattice $\oplat$ of $\Orth(E)$ is, therefore, also a regular vector sublattice of $\regops(E)$, and \cref{res:oLtop_on_vector_lattice_of_operators} then shows how \oLtops\ on $\oplat$ can be obtained from an \oLtop\ on $E$ as (absolute) strong operator topologies. In particular, this makes the fact that part~(1) of \cref{res:oLtop_and_uoLtop_on_E_and_Orth(E)} implies its part~(2) more concrete. The fact that part~(1) implies part~(2) is made more concrete as a special case of the following consequence of \cref{res:uoLtops_on_lattices_of_operators_most_precise}.

\begin{theorem}\label{res:uoLtops_on_orthomorphism_most_precise}
	Let $E$ be a Dedekind complete vector lattice. Suppose that $E$ admits an \oLtop\ $\tau_E$. Take a regular vector sublattice  $\oplat$ of $\Orth(E)$, a non-empty subset $\opset$ of $\oplat$, and a non-empty subset $S$ of $E$.
	
	Then $\unb_\opset\optop{S}\tau_E$ is a \uoLtop\ on $\oplat$.
	
	We let $I_S$ denote the ideal of $E$ that is generated by $S$, and $I_\opset$ the ideal of $\oplat$ that is generated by $\opset$. For a net $\opnet$ in $\oplat$, the following are equivalent:
	\begin{enumerate}
		\item $T_\alpha\conv{\unb_\opset\optop{S}\tau_E}0$;
		\item $\abs{T_\alpha s}\wedge \abs{T s}\conv{\tau_E} 0$ for all $T\in\opset$ and $s\in S$;
		\item $\abs{T_\alpha x}\wedge \abs{Tx}\conv{\tau_E} 0$ for all $T\in I_\opset$ and $x\in I_S$.	
	\end{enumerate}
	
	Suppose that $\tau_E$ is actually a \emph{Hausdorff} \oLtop\ on $E$. Then the following are equivalent:
	\begin{enumerate}
		\item[(a)] $\unb_\opset\optop{S}\tau_E$ is a \necun\ \emph{Hausdorff} \uoLtop\ on~$\oplat$;
		\item[(b)] $I_S$ separates the points of $\oplat$ and $I_\opset$ is order dense in $\oplat$.
	\end{enumerate}
	
	In that case, the Hausdorff \uoLtop\ $\unb_\opset\optop{S}\tau_E$ on $\oplat$ is the restriction of the \necun\ Hausdorff \uoLtop\ on $\regops(E,F)$, i.e., of $\unb_{\regops(E,F)}\optop{E}\tau_E$, and the criteria in~\uppars{1},~\uppars{2}, and~\uppars{3} are also equivalent to:
	\begin{enumerate}
		\item[(4)] $(\abs{T_\alpha}\wedge \abs{T})x\conv{\uoLt_E} 0$ for all $T\in\regops(E)$ and $x\in E$.
	\end{enumerate}
\end{theorem}

\section{Comparing uniform and strong convergence structures on $\Orth(E)$} \label{sec:uniform_and_strong_convergence_structures_on_orth}

\noindent Let $E$ and $F$ be vector lattices, where $F$ is Dedekind complete, and let $\opnet$ be a net in $\regops(E,F)$. In   \cref{sec:uniform_and_strong_convergence_structures_on_regops}, we studied the relation between uniform and strong convergence of $\opnet$ for order convergence, unbounded order convergence, and\textemdash when applicable\textemdash convergence in the Hausdorff \uoLtop. In the present section, we consider the case where $\opnet$ is actually contained in $\Orth(E)$. As we shall see, the relation between uniform and strong convergence is now much more symmetrical than in the general case of \cref{sec:uniform_and_strong_convergence_structures_on_regops}; see \cref{res:ordertoorderinorth} (and \cref{res:oRKF}), \cref{res:uoRKF}, and \cref{res:uoLt_RKF}, below.

These positive results might, perhaps, lead one to wonder whether some of the three uniform convergence structures under consideration might actually even be identical for the orthomorphisms. This, however, is not the case. There even exist sequences of positive orthomorphisms on separable reflexive Banach lattices with weak order units showing that the two `reverse' implications in question are not generally valid. For this, we consider $E\coloneqq \Ell_p([0,1])$ for $1<p<\infty$. In that case, $\Orth(E)$ can canonically be identified with $\Ell_\infty([0,1])$ as an $f$-algebra; see \cite[Example~2.67]{aliprantis_burkinshaw_POSITIVE_OPERATORS_SPRINGER_REPRINT:2006}, for example. The uo-convergence of a net in the regular vector sublattice $\Ell_\infty([0,1])$ of $\Ell_0([0,1])$ coincides with that in $\Ell_0([0,1])$ which, according to \cite[Proposition~3.1]{gao_troitsky_xanthos:2017}, is simply convergence almost everywhere in the case of sequences. According to \cite[Theorem~6.3]{deng_de_jeu_UNPUBLISHED:2020a}, the convergence of a net in the Hausdorff \uoLtop\ of $\Ell_\infty([0,1])$ is equal to the convergence in measure. For $n\geq 1$, set $f_n\coloneqq n\chi_{[0,1/n]}$. Then $f_n\uoconv 0$ in $\Ell_\infty([0,1])$, but it is not true that $f_n\oconv 0$ in $\Ell_\infty([0,1])$ since the $f_n$ are not even order bounded in $\Ell_\infty([0,1])$. Using $\chi_{[(k-1)2^{-n},k2^{-n}]}$ for $n\geq 1$ and $k=1,\dotsc,2^n$, one easily finds a sequence that is convergent to zero in measure, but that is not convergent in any point of $[0,1]$.

We now start with uniform and strong order convergence for nets of orthomorphisms. For this, we need a few preparatory results. The first one is on general order continuous operators.

\begin{lemma} \label{res:ideal_to_band}
	Let $E$ be a Dedekind complete vector lattice, let $\opnet$ be a decreasing net in $\pos{\ocops(E)}$, and let $F$ be an order dense vector sublattice of $E$.  The following are equivalent:
	\begin{enumerate}
	\item $T_\alpha x\oconv 0$ in $E$ for all $x\in F$;
	\item $T_\alpha x\oconv 0$ in $E$ for all $x\in E$.
	\end{enumerate}
\end{lemma}

\begin{proof} We need to show only that part~(1) implies part~(2). Let $T\in \obops(E)$ be such that $T_\alpha \downarrow T$ in $\obops(E)$. Then $T\in\pos{\ocops(E)}$. The hypothesis under part~(1) and \cite[Theorem~1.18]{aliprantis_burkinshaw_POSITIVE_OPERATORS_SPRINGER_REPRINT:2006} imply that $Tx=0$ for all $x\in\pos{F}$. Since $F$ is order dense in $E$ and $T$ is order continuous, it now follows from \cite[Theorem~1.34]{aliprantis_burkinshaw_POSITIVE_OPERATORS_SPRINGER_REPRINT:2006} that $T=0$. Using \cite[Theorem~1.18]{aliprantis_burkinshaw_POSITIVE_OPERATORS_SPRINGER_REPRINT:2006} once more, we conclude that $T_\alpha x\downarrow 0$ in $E$ for all $x\in\pos{E}$, and the statement in part~(2) follows.
\end{proof}

\begin{proposition}\label{res:monotone_order_orth}
	Let $E$ be a Dedekind complete vector lattice, let $\opnet$ be a decreasing net in $\pos{\Orth(E)}$, and let $S$ be a non-empty subset of $E$. The following are equivalent:
	\begin{enumerate}
		\item $T_\alpha s\oconv 0$ in $E$ for all $s \in S$;
		\item $T_\alpha x\oconv 0$ in $E$ for all $x \in B_S$.
	\end{enumerate}
In particular, if $E$ has a positive weak order unit $e$, then $T_\alpha x\oconv 0$ in $E$ for all $x\in E$ if and only if $T_\alpha e\downarrow 0$ in $E$.
\end{proposition}

\begin{proof} We need to show only that part~(1) implies part~(2). Take $y\in \pos{I}_S$. There exist $s_1,\dotsc,s_n\in S$ and $\lambda_1,\dotsc,\lambda_n\geq 0$ such that $0\leq y\leq\sum_{i=1}^n \lambda_i \abs{s_i}$. Hence $0\leq T_\alpha y\leq\sum_{i=1}^n \lambda_i T_\alpha \abs{s_i}=\sum_{i=1}^n \lambda_i \abs{T_\alpha s_i}$ for $\alpha\in\is$, and the assumption then implies that $T_\alpha y\downarrow 0$ in $E$. 
	Since orthomorphisms preserve bands, we have $T_\alpha y\in B_S$ for all $\alpha\in\is$, and the fact that $B_S$ is an ideal of $E$ now shows that $T_\alpha y\downarrow 0$ in $B_S$. It follows that $T_\alpha y\oconv 0$ in $B_S$ for all $y\in I_S$.  Since the restriction of each $T_\alpha$ to the regular vector sublattice $B_S$ of $E$ is again order continuous, and since $I_S$ is an order dense vector sublattice of the vector lattice $B_S$, \cref{res:ideal_to_band} implies that $T_\alpha y\oconv 0$ in $B_S$ for all $y\in B_S$. The fact that $B_S$ is a regular vector sublattice of $E$ then yields that $T_\alpha y\oconv 0$ in $E$ for all $y\in B_S$.
\end{proof}

\begin{lemma}\label{res:lemmaRKF}
	
	Let $E$ be a Dedekind complete vector lattice, and let $\mathcal S$ be a subset of $\Orth(E)$ that is bounded above in $\regops(E)$. Then, for $x\in\pos{E}$,
	\[
	\left(\bigvee_{T\in \mathcal S}T\right) x=\bigvee_{T\in \mathcal S} Tx
	\]
\end{lemma}

\begin{proof}
	Using \cite[Theorem~1.67.(b)]{aliprantis_burkinshaw_LOCALLY_SOLID_RIESZ_SPACES_WITH_APPLICATIONS_TO_ECONOMICS_SECOND_EDITION:2003} in the second step, we see that, for $x\in\pos{E}$,
	\[
	\left(\bigvee_{T\in \mathcal S}T\right) x=\left(\bigvee_{T^\vee\in \mathcal S^\vee}T^\vee\right)x=\bigvee_{T^\vee\in \mathcal S^\vee}T^\vee x.
	\]
	By \cref{eq:sup_of_two_orthormorphisms}, this equals
	\[
	\bigvee_{y^\vee\in (\mathcal S x)^\vee}y^\vee=\bigvee_{y\in \mathcal Sx} y=\bigvee_{T\in \mathcal S} Tx.
	\]
	\end{proof}

We can now establish our main result regarding uniform and strong order convergence for nets of orthomorphisms.

\begin{theorem}\label{res:ordertoorderinorth}
	
	Let $E$ be a Dedekind complete vector lattice, and let $\opnet$ be a net in $\Orth(E)$ that is order bounded in $\regops(E)$. Let $S$ be a non-empty subset of $E$ with $B_S=E$. The following are equivalent:
	\begin{enumerate}
		\item $T_\alpha\oconv 0$ in $\Orth(E)$;
		\item $T_\alpha\oconv 0$ in $\regops(E)$;
		\item $T_\alpha s\oconv 0$ in $E$ for all $s\in S$;
		\item $T_\alpha x\oconv 0$ in $E$ for all $x\in E$.
	\end{enumerate}
	In particular, when $E$ has a weak order unit $e$, then $T_\alpha\oconv 0$ in $\regops(E)$ if and only if $T_\alpha e\oconv 0$ in $E$.
\end{theorem}

As for \cref{res:UBP}, the order boundedness of the net could equivalently have been required in $\Orth(E)$.

\begin{proof} Since $\opnet$ is supposed to be order bounded in the regular vector sublattice $\Orth(E)$, the equivalence of the parts~(1) and~(2) follows from \cite[Corollary~2.12]{gao_troitsky_xanthos:2017}. \cref{res:o_to_o} shows that part~(2) implies part~(4), and evidently part~(4) implies part~(3). The proof will be completed by showing that part~(3) implies part~(1). Suppose that $T_\alpha s\oconv 0$ in $E$ for all $s\in S$ or, equivalently, that $\abs{T_\alpha}\abs{s}=\abs{T_\alpha s}\oconv 0$ in $E$ for all $s\in S$. For $\alpha\in\is$, set $\widetilde T_\alpha\coloneqq\bigvee_{\beta\geq\alpha}\abs{T_\beta}$ in $\regops(E)$.  Since \cref{res:lemmaRKF} shows that $\widetilde T_\alpha\abs{s}=\bigvee_{\beta\geq\alpha}\abs{T_\beta}\abs{s}$ for $\alpha\in \is$ and $s\in S$, we see that $\widetilde T_\alpha\abs{s}\downarrow 0$ in $E$ for all $s\in S$. \cref{res:monotone_order_orth} then yields that $\widetilde T_\alpha x\oconv 0 $ for all $x\in B_{\abs{S}}=E$. Using that $\widetilde T_\alpha\downarrow$, it follows that $\widetilde T_\alpha\downarrow 0$ in $\regops(E)$. Since $\abs{T_\alpha}\leq \widetilde{T}_\alpha$ for $\alpha\in\is$, we see that $\abs{T_\alpha}\oconv 0$ in $\regops(E)$, as required.
\end{proof}

In view of \cref{res:o_to_o}, the most substantial part of \cref{res:ordertoorderinorth} is the fact that the parts~(3) and~(4) imply the parts~(1) and~(2). For this to hold in general, the assumption that $\opnet$ be order bounded is actually necessary. To see this, let $\Gamma$ be an uncountable set that is supplied with the counting measure, and consider $E\coloneqq\ell_p(\Gamma)$ for $1\leq p<\infty$. Set
\[
\is\coloneqq\{\,(n,S): n\geq 1,\,S\subset\Gamma\text{ is at most countably infinite}\,\}
\]
and, for $(n_1,S_1),\,(n_2,S_2)\in\is$, say that $(n_1,S_2)\leq (n_2,S_2)$ when $n_1\leq n_2$ and $S_1\subseteq S_2$. For $(n,S)\in\is$, define $T_{(n,S)}\in\centre(E)=\Orth(E)$ by setting
\[
T_{(n,S)}x\coloneqq n\chi_{\Gamma\setminus S}x
\]
for all $x:\Gamma\to\RR$ in $E$. Take an $x\in E$. Then the net $\netgen{T_{(n,S)}x}{(n,S)\in\is}$ has a tail
 $\netgen{T_{(n,S)}x}{(n,S)\geq (1,\mathrm{supp\,} x)}$ that is identically zero. Hence $T_{(n,S)}x\oconv 0$ in $E$ for all $x\in E$. We claim that $\netgen{T_{(n,S)}}{(n,S)\in\is}$ is not order convergent in $\Orth(E)$, and not even in $\obops(E)$. For this, it is sufficient to show that it does not have any tail that is order bounded in $\obops(E)$. Suppose, to the contrary, that there exist an $n_0\geq 1$, an at most countably infinite subset $S_0$ of $\Gamma$, and a $T\in\obops(E)$ such that $T_{(n,A)}\leq T$ for all $(n,A)\in\is$ with $n\geq n_0$ and $A\supseteq A_0$. As $\Gamma$ is uncountable, we can choose an $x_0\in\Gamma\setminus A_0$; we let $e_{x_0}$ denote the corresponding atom in $E$. Then, in particular, $T_{(n,A_0)} e_{x_0}\leq T e_{x_0}$ for all $n\geq n_0$. Hence $Te_{x_0}\geq n e_{x_0}$ for all $n\geq n_0$, which is impossible.

\smallskip

Using \cref{res:ordertoorderinorth} and \cref{res:coro_Tn}, the following is easily established. In contrast to \cref{res:ordertoorderinorth}, there is no order boundedness in the hypotheses because this is taken care of by \cref{res:coro_Tn}.

\begin{theorem}\label{res:oRKF}
	Let $E$ be a Dedekind complete vector lattice, and let $\seq{T_n}{n}$ be a sequence in $\Orth(E)$. Let $S$ be a non-empty subset of $E$ such that $I_S=E$. The following are equivalent:
\begin{enumerate}
	\item $T_n\oconv 0$ in $\Orth(E)$;
	\item $T_n\oconv 0$ in $\regops(E)$;
	\item $T_n s\oconv 0$ in $E$ for all $s\in S$;
	\item $T_n x\oconv 0$ in $E$ for all $x\in E$.
\end{enumerate}
In particular, when $E$ has a strong order unit $e$, then $T_n\oconv 0$ in $\Orth(E)$ if and only if $T_n e\oconv 0$ in $E$.	
\end{theorem}

\begin{remark}\label{remark_only_ideal} Even for Banach lattices with order continuous norms, the condition that $I_S=E$ in \cref{res:oRKF}, cannot be relaxed to $B_S=E$ as in \cref{res:ordertoorderinorth}. To see this, we choose $E\coloneqq c_0$ and set $e\coloneqq\bigvee_{n\geq 1} e_i/i^2$, where $\seq{e_i}{i}$ is the standard unit basis of $E$. It is clear that $B_e=E$. For $n\geq 1$, there exists a unique $T_n\in \Orth(E)$ such that, for $i\geq 1$, $T_ne_i= ne_i$ when $i=n$, and $T_n e_i=0$ when $i\neq n$. It is clear that $T_n e\xrightarrow{o} 0$ in $E$. However, a consideration of $T_n\left(\bigvee_{i\geq 1}e_i/i\right)$ for $n\geq 1$ shows that $\seq{T_n}{n}$ fails to be order bounded in $\Orth(E)$, hence cannot be order convergent in $\Orth(E)$.
\end{remark}

We continue our comparison of uniform and strong convergence structures on the orthomorphisms by considering unbounded order convergence. In that case, the result is as follows.

\begin{theorem}\label{res:uoRKF} Let $E$ be a Dedekind complete vector lattice, and let $\opnet$ be a net in $\Orth(E)$. Let $S$ be a non-empty subset of $E$ such that $B_S=E$. The following are equivalent:
	\begin{enumerate}
		\item $T_\alpha\uoconv 0$ in $\Orth{E}$;
		\item $T_\alpha\uoconv 0$ in $\regops(E)$;
		\item $T_\alpha s\uoconv 0$ in $E$ for all $s\in S$;
		\item $T_\alpha x\uoconv 0$ in $E$ for all $x\in E$.
	\end{enumerate}
	In particular, when $E$ has a weak order unit $e$, then $T_\alpha\uoconv 0$ in $\Orth(E)$ if and only if $T_\alpha e\uoconv 0$ in $E$.
\end{theorem}

\begin{proof} Since $\Orth(E)$ is a regular vector sublattice of $\regops(E)$, the equivalence of the parts~(1) and~(2) is clear from \cite[Theorem~3.2]{gao_troitsky_xanthos:2017}

We prove that part~(2) implies part~(4). Suppose that $T_\alpha\uoconv 0$ in $\regops(E)$, so that, in particular, $\abs{T_\alpha}\wedge \Id\oconv 0$ in $\regops(E)$. Take $x\in E$. Using \cref{eq:sup_of_two_orthormorphisms} in the second step, and  \cref{res:o_to_o} in the third, we have
\[
(\abs{T_\alpha} \abs{x})\wedge \abs{x}=(\abs{T_\alpha} \abs{x})\wedge (\Id \abs{x})=(\abs{T_\alpha}\wedge \Id) \abs{x}\oconv 0.
\]
Since the net $\netgen{\abs{T_\alpha}\abs{x}}{\alpha\in\is}$ is contained in the band $B_{\abs{x}}$, it now follows from \cite[Proposition~7.4]{deng_de_jeu_UNPUBLISHED:2020a} that $\abs{T_\alpha}\abs{x}\uoconv 0$ in $E$. As $\abs{T_\alpha}\abs{x}=\abs{T_\alpha x}$, we conclude that $T_\alpha x\uoconv 0$ in $E$.

It is clear that part~(4) implies part~(3).

We prove that part~(3) implies part~(2). Suppose that $T_\alpha s\uoconv 0$ in $E$ for all $s\in S$, so that also  $\abs{T_\alpha} \abs{s}=\abs{T_\alpha s}\uoconv 0$ in $E$ for $s\in S$. Using \cref{eq:sup_of_two_orthormorphisms} again, we have
	\[
	(\abs{T_\alpha} \wedge \Id)\abs{s}=(\abs{T_\alpha} \abs{s})\wedge \abs{s}\oconv 0
	\]
	in $E$ for all $x\in S$. In view of the order boundedness of $\netgen{\abs{T_\alpha}\wedge \Id}{\alpha \in\is}$, \cref{res:ordertoorderinorth} then yields that $\abs{T_\alpha} \wedge \Id\oconv 0$ in $\regops(E)$. As $\Id$ is a weak order unit of $\Orth(E)$, \cite[Lemma 3.2]{gao_xanthos:2014} (or the more general \cite[Proposition~7.4]{deng_de_jeu_UNPUBLISHED:2020a}) shows that $T_\alpha\uoconv 0$ in $\regops(E)$.
\end{proof}

We now consider uniform and strong convergence of nets of orthomorphisms for the Hausdorff \uoLtop. Let $E$ be a Dedekind complete vector lattice. Suppose that $E$ admits a \necun\ Hausdorff \uoLtop\ $\uoLt_E$. We recall from \cref{res:uoLtops_on_orthomorphism_most_precise} that $\Orth(E)$ then also admits a \necun\ Hausdorff \uoLtop\ $\uoLt_{\Orth(E)}$, and that this topology equals $\unb_{\Orth(E)}\optop{E}\uoLt_E$. Furthermore, for a net $\opnet$ in $\Orth(E)$, we have that $T_\alpha\conv{\uoLt_{\Orth(E)}}0$ if and only if $\abs{T_\alpha x}\wedge\abs{Tx}\conv{\uoLt_E}0$ for all $T\in\Orth(E)$ and $x\in E$.

We need two preparatory results.

\begin{lemma}\label{res:uoLtweakunit}
	Let $E$ be a vector lattice that admits a \necun\ Hausdorff \uoLtop\ $\uoLt_E$. Suppose that $E$ has a positive weak order unit $e$. Let $\net$ be a net in $E$. Then $x_\alpha\conv{\uoLt_E} 0$ in $E$ if and only if $\abs{x_\alpha}\wedge e\conv{\uoLt_E} 0$ in $E$.
\end{lemma}

\begin{proof} We need to show only the ``if''-part. Suppose that $\abs{x_\alpha}\wedge e\conv{\uoLt_E} 0$ in $E$. For each $x\in E$, there exists a net $\netgen{y_\beta}{\beta\in\istwo}$ in $I_e$ such that $y_\beta\oconv x$, and then certainly $y_\beta\conv{\uoLt_E} x$. Hence $\overline{I_e}^{\uoLt_E}=E$. An appeal to \cite[Proposition~9.8]{taylor:2019} then shows that $x_\alpha\conv{\unb_{E}\uoLt_E}0$. Since $\unb_E\uoLt_E=\uoLt_E$, we are done.	
\end{proof}

Our second preparatory result is in the same vein as \cref{res:tauctn}.

\begin{lemma}\label{res:oLt_preparation}
	Let $E$ be a vector lattice with the principal projection property that admits a \uppars{not necessarily Hausdorff} \oLtop\ $\tau_E$, and let $\opnet$ be a net in $\Orth(E)$. Let $S$ be a non-empty subset of $E$ such that $B_S=E$. Suppose that $T_\alpha s\conv{\tau_E}0$ for all $s\in S$. Then $T_\alpha x\conv{\unb_E\tau_E}0$ for all $x\in E$.
\end{lemma}

\begin{proof}
Using \cref{eq:modulus_of_orthormorphism}, it follows easily that $T_\alpha x\conv{\tau_E}0$ for all $x\in I_S$. Take an $x\in E$, and let $U$ be a solid $\tau_E$-neighbourhood $U$ of 0. Choose a  $\tau_E$-neighbour\-hood $V$ of $0$ such that $V+V\subseteq U$. There exists a net $\nettwo$ in $I_S$ such that $x_\beta\oconv x$ in $E$, and then we can choose a $\beta_0\in\istwo$ such that $\abs{x-x_{\beta_0}}\in V$. As $\abs{T_\alpha}\abs{x_{\beta_0}}=\abs{T_\alpha x_{\beta_0}}\conv{\tau_E}0$, there exists an $\alpha_0\in\is$ such that $\abs{T_\alpha}\abs{x_{\beta_0}}\in V$ for all $\alpha\geq\alpha_0$. For all $\alpha\geq\alpha_0$, we then have
\begin{align*}
0&\leq (\abs{T_\alpha x} )\wedge\abs{x}\\
&=(\abs{T_\alpha}\wedge\Id)\abs{x}\\
&\leq (\abs{T_\alpha}\wedge\Id)\abs{x_{\beta_0}}+ (\abs{T_\alpha}\wedge\Id)\abs{x - x_{\beta_0}}\\
&\leq \abs{T_\alpha}\abs{x_{\beta_0}} + \abs{x-x_{\beta_0}}\\
&\in V+V\subseteq U.
\end{align*}
As $U$ is solid, we see that $(\abs{T_\alpha x} )\wedge\abs{x}\in U$ for $\alpha\geq\alpha_0$, and we conclude that $(\abs{T_\alpha x})\wedge\abs{x}\conv{\tau_E}0$. Since $\abs{T_\alpha x}\in B_{\abs{x}}$ for $\alpha\in\is$, it then follows from \cite[Proposition~9.8]{taylor:2019} that $\abs{T_\alpha x}\wedge \abs{y}\conv{\tau_E}0$ in $E$ for all $y\in B_{\abs{x}}$. As $B_{\abs{x}}$ is a projection band in $E$, this holds, in fact, for all $y\in E$.

\end{proof}

\begin{theorem}\label{res:uoLt_RKF} Let $E$ be a Dedekind complete vector lattice. Suppose that $E$ admits a \necun\ Hausdorff \uoLtop\ $\uoLt_E$, so that $\Orth(E)$ and $\regops(E)$ also admit \necun\ Hausdorff \uoLtops\ $\uoLt_{\Orth(E)}$ and $\uoLt_{\regops(E)}$, respectively. Let $\opnet$ be a net in $\Orth(E)$. Let $S$ be a non-empty subset $S$ of $E$ such that $B_S=E$. The following are equivalent:
	\begin{enumerate}
		\item $T_\alpha\conv{\uoLt_{\Orth(E)}} 0$ in $\Orth(E)$;\label{part:uoLt_RKF_1}
		\item $T_\alpha\conv{\uoLt_{\regops(E)}} 0$ in $\regops(E)$;\label{part:uoLt_RKF_3}
		\item $T_\alpha s\conv{\uoLt_E} 0$ in $E$ for all $s\in S$;\label{part:uoLt_RKF_4}
		\item $T_\alpha x\conv{\uoLt_E} 0$ in $E$ for all $x\in E$.\label{part:uoLt_RKF_2}
	\end{enumerate}
In particular, when $E$ has a weak order unit $e$, then $T_\alpha\conv{\uoLt_{\Orth(E)}}0$ in $\Orth(E)$ if and only if $T_\alpha e\conv{\uoLt_E} 0$ in $E$.	

\end{theorem}

\begin{proof}

The equivalence of the parts \partref{part:uoLt_RKF_1} and~\partref{part:uoLt_RKF_3} follows from the final part of \cref{res:uoLtops_on_lattices_of_operators_most_precise}.

	We prove that part~\partref{part:uoLt_RKF_1} implies part~\partref{part:uoLt_RKF_2}. Suppose that $T_\alpha\conv{\uoLt_{\Orth(E)}} 0$ in $\Orth(E)$. Take an $x\in E$. Then certainly $\abs{T_\alpha x}\wedge\abs{x}=\abs{T_\alpha x}\wedge \abs{\Id x}\conv{\uoLt_E}0$. The net $\netgen{T_\alpha x}{\alpha\in\is}$ is contained in the band $B_{\abs{x}}$. Since, by \cite[Proposition~5.12]{taylor:2019}, the regular vector sublattice $B_{\abs{x}}$ of $E$ also admits a \necun\ Hausdorff \uoLtop\ (namely, the restriction of $\uoLt_E$ to $B_{\abs{x}}$),  it then follows from \cref{res:uoLtweakunit} that $T_\alpha x\conv{\uoLt_E}0$ in $E$.
	
	We prove that part~\partref{part:uoLt_RKF_2} implies part~\partref{part:uoLt_RKF_1}. Suppose that $T_\alpha x\conv{\uoLt_E}0$ for all $x\in E$. Since $\uoLt_E$ is locally solid, we then also have $\abs{T_\alpha x}\wedge{\abs{Tx}}\conv{\uoLt_E}0$ for all $T\in\Orth(E)$ and $x\in E$. Hence $T_\alpha\conv{\uoLt_{\Orth(E)}} 0$ in $\Orth(E)$.
	
	It is clear that part~\partref{part:uoLt_RKF_2} implies  part~\partref{part:uoLt_RKF_4}.
	
	Since $\unb_E\uoLt_E=\uoLt_E$, \cref{res:oLt_preparation} shows that part~\partref{part:uoLt_RKF_4} implies  part~\partref{part:uoLt_RKF_2}.
\end{proof}

%%%%%%%%%%%%%%%%%%%%%%%%%%%%%%%%% ACKNOWLEDGEMENTS %%%%%%%%%%%%%%%%%%%%%%%%%%%%%%%%%%%%%

\subsection*{Acknowledgements} During this research, the first author was supported by a grant of China Scholarship Council (CSC). The authors thank the anonymous referees for their careful reading of the manuscript and their constructive suggestions. These have led, amongst others, to \cref{res:UBP} now being established for general vector lattices rather than for Banach lattices as in the original manuscript, and with an easier proof.

\bibliographystyle{plain}
\urlstyle{same}
\bibliography{general_bibliography}

\def\cprime{$'$} \def\cprime{$'$} \def\cprime{$'$} \def\cprime{$'$}
  \def\cprime{$'$} \def\cprime{$'$} \def\cprime{$'$}
  \def\lfhook#1{\setbox0=\hbox{#1}{\ooalign{\hidewidth
  \lower1.5ex\hbox{'}\hidewidth\crcr\unhbox0}}} \def\cprime{$'$}
  \def\cprime{$'$}
\begin{thebibliography}{10}

\bibitem{aliprantis_border_INFINITE_DIMENSIONAL_ANALYSIS_THIRD_EDITION:2006}
C.D. Aliprantis and K.C. Border.
\newblock {\em Infinite dimensional analysis. A hitchhiker's guide.}
\newblock Springer, Berlin, third edition, 2006.

\bibitem{aliprantis_burkinshaw_LOCALLY_SOLID_RIESZ_SPACES_WITH_APPLICATIONS_TO_ECONOMICS_SECOND_EDITION:2003}
C.D. Aliprantis and O.~Burkinshaw.
\newblock {\em Locally solid {R}iesz spaces with applications to economics},
  volume 105 of {\em Mathematical Surveys and Monographs}.
\newblock American Mathematical Society, Providence, RI, second edition, 2003.

\bibitem{aliprantis_burkinshaw_POSITIVE_OPERATORS_SPRINGER_REPRINT:2006}
C.D. Aliprantis and O.~Burkinshaw.
\newblock {\em Positive operators}.
\newblock Springer, Dordrecht, 2006.
\newblock Reprint of the 1985 original.

\bibitem{beattie_butzmann_CONVERGENCE_STRUCTURES_AND_APPLICATIONS_TO_FUNCTIONAL_ANALYSIS:2002}
R.~Beattie and H.-P. Butzmann.
\newblock {\em Convergence structures and applications to functional analysis}.
\newblock Kluwer Academic Publishers, Dordrecht, 2002.

\bibitem{bogachev_MEASURE_THEORY_VOLUME_II:2007}
V.I. Bogachev.
\newblock {\em Measure theory. {V}ol. {II}}.
\newblock Springer-Verlag, Berlin, 2007.

\bibitem{conradie:2005}
J.~Conradie.
\newblock The coarsest {H}ausdorff {L}ebesgue topology.
\newblock {\em Quaest. Math.}, 28(3):287--304, 2005.

\bibitem{conway_A_COURSE_IN_FUNCTIONAL_ANALYSIS_SECOND_EDITION:1990}
J.B. Conway.
\newblock {\em A course in functional analysis}, volume~96 of {\em Graduate
  Texts in Mathematics}.
\newblock Springer-Verlag, New York, second edition edition, 1990.

\bibitem{dabboorasad_emelyanov_marabeh_UNPUBLISHED:2017}
Y.A. Dabboorasad, E.Y. Emelyanov, and M.A.A Marabeh.
\newblock Order convergence in inifinite-dimensional vector lattices is not
  topological.
\newblock Preprint, 2017. Online at https://arxiv.org/pdf/1705.09883.pdf.

\bibitem{de_jeu_wortel:2014}
M.~de~Jeu and M.~Wortel.
\newblock Compact groups of positive operators on {B}anach lattices.
\newblock {\em Indag. Math. (N.S.)}, 25(2):186--205, 2014.

\bibitem{deng_de_jeu_UNPUBLISHED:2020c}
Y.~Deng and M.~de~Jeu.
\newblock Convergence structures and {H}ausdorff uo-{L}ebesgue topologies on
  vector lattice algebras of operators.
\newblock Preprint, 2020. Online at \url{https://arxiv.org/pdf/2011.03768.pdf}.

\bibitem{deng_de_jeu_UNPUBLISHED:2020a}
Y.~Deng and M.~de~Jeu.
\newblock Vector lattices with a {H}ausdorff uo-{L}ebesgue topology.
\newblock Preprint, 2020. Online at \url{https://arxiv.org/pdf/2005.14636.pdf}.

\bibitem{dudley:1964}
R.M. Dudley.
\newblock On sequential convergence.
\newblock {\em Trans. Amer. Math. Soc.}, 112:483--507, 1964.

\bibitem{folland_REAL_ANALYSIS_SECOND_EDITION:1999}
G.B. Folland.
\newblock {\em Real analysis. {M}odern techniques and their applications}.
\newblock Pure and Applied Mathematics. John Wiley \& Sons, Inc., New York,
  second edition, 1999.

\bibitem{gao_troitsky_xanthos:2017}
N.~Gao, V.G. Troitsky, and F.~Xanthos.
\newblock Uo-convergence and its applications to {C}es\`aro means in {B}anach
  lattices.
\newblock {\em Israel J. Math.}, 220(2):649--689, 2017.

\bibitem{gao_xanthos:2014}
N.~Gao and F.~Xanthos.
\newblock Unbounded order convergence and application to martingales without
  probability.
\newblock {\em J. Math. Anal. Appl.}, 415(2):931--947, 2014.

\bibitem{hackenbroch:1977}
W.~Hackenbroch.
\newblock Representation of vector lattices by spaces of real functions.
\newblock In {\em Functional analysis: surveys and recent results ({P}roc.
  {C}onf., {P}aderborn, 1976)}, pages 51--72. North--Holland Math. Studies,
  Vol. 27; Notas de Mat., No. 63. North-Holland, Amsterdam, 1977.

\bibitem{husain_INTRODUCTION-TO_TOPOLOGICAL_GROUPS:1966}
T.~Husain.
\newblock {\em Introduction to topological groups}.
\newblock W. B. Saunders Co., Philadelphia, Pa.-London, 1966.

\bibitem{kandic_taylor:2018}
M.~Kandi\'{c} and M.A. Taylor.
\newblock Metrizability of minimal and unbounded topologies.
\newblock {\em J. Math. Anal. Appl.}, 466(1):144--159, 2018.

\bibitem{li_chen:2018}
H.~Li and Z.~Chen.
\newblock Some loose ends on unbounded order convergence.
\newblock {\em Positivity}, 22:83--90, 2018.

\bibitem{luxemburg_zaanen_RIESZ_SPACES_VOLUME_I:1971}
W.A.J. Luxemburg and A.C. Zaanen.
\newblock {\em Riesz spaces. {V}ol. {I}}.
\newblock North-Holland Publishing Co., Amsterdam-London; American Elsevier
  Publishing Co., New York, 1971.

\bibitem{meyer-nieberg_BANACH_LATTICES:1991}
P.~Meyer-Nieberg.
\newblock {\em Banach lattices}.
\newblock Universitext. Springer-Verlag, Berlin, 1991.

\bibitem{schaefer_wolff_arendt:1978}
H.H. Schaefer, M.~Wolff, and W.~Arendt.
\newblock On lattice isomorphisms with positive real spectrum and groups of
  positive operators.
\newblock {\em Math. Z.}, 164(2):115--123, 1978.

\bibitem{taylor_THESIS:2018}
M.A. Taylor.
\newblock Unbounded convergences in vector lattices.
\newblock Master's thesis, University of Alberta, University of Alberta,
  Edmonton, 2018.
\newblock Online at
  \url{https://era.library.ualberta.ca/items/2cc15102-1119-421c-8228-1ad1fb96d357}.

\bibitem{taylor:2019}
M.A. Taylor.
\newblock Unbounded topologies and {$uo$}-convergence in locally solid vector
  lattices.
\newblock {\em J. Math. Anal. Appl.}, 472(1):981--1000, 2019.

\bibitem{wickstead:1977a}
A.W. Wickstead.
\newblock Representation and duality of multiplication operators on
  {A}rchimedean {R}iesz spaces.
\newblock {\em Compositio Math.}, 35(3):225--238, 1977.

\bibitem{zaanen_INTRODUCTION_TO_OPERATOR_THEORY_IN_RIESZ_SPACES:1997}
A.C. Zaanen.
\newblock {\em Introduction to operator theory in {R}iesz spaces}.
\newblock Springer-Verlag, Berlin, 1997.

\end{thebibliography}
%\bibliography{../../tex/templates/bibliography/general_bibliography}

\end{document}